\documentclass{amsart}

\usepackage{mathtools}
\usepackage{mathrsfs}
\usepackage{stmaryrd}
\usepackage{enumerate}
\usepackage{tikz-cd}
\usepackage[all]{xy}
\usepackage{aliascnt}

\usepackage{shuffle}
\usepackage{ytableau}
\newcommand{\blank}{\phantom{2}}

\usepackage[colorlinks=true, linkcolor=blue, citecolor=blue]{hyperref}
\usepackage{fullpage}
\usepackage{verbatim}
\usepackage{amssymb}

\newtheorem{theorem}{Theorem}[section]

\newaliascnt{lemma}{theorem}
\newtheorem{lemma}[lemma]{Lemma}
\aliascntresetthe{lemma}

\newaliascnt{corollary}{theorem}
\newtheorem{corollary}[corollary]{Corollary}
\aliascntresetthe{corollary}

\newaliascnt{proposition}{theorem}
\newtheorem{proposition}[proposition]{Proposition}
\aliascntresetthe{proposition}

\newaliascnt{definitionlemma}{theorem}

\aliascntresetthe{definitionlemma}

\newaliascnt{conjecture}{theorem}

\aliascntresetthe{conjecture}

\newaliascnt{question}{theorem}

\aliascntresetthe{question}

\theoremstyle{definition}

\newaliascnt{definition}{theorem}
\newtheorem{definition}[definition]{Definition}
\aliascntresetthe{definition}

\newaliascnt{remark}{theorem}
\newtheorem{remark}[remark]{Remark}
\aliascntresetthe{remark}

\newaliascnt{example}{theorem}
\newtheorem{examplex}[example]{Example}
\newenvironment{example}
  {\pushQED{\qed}\examplex}
  {\popQED\endexamplex}
\aliascntresetthe{example}

\newaliascnt{notation}{theorem}

\aliascntresetthe{notation}

\definecolor{darkblue}{rgb}{0.6,0,0.1}

\newcommand{\newword}[1]{\textcolor{darkblue}{\textbf{\emph{#1}}}}

\usepackage{tikz}
\usetikzlibrary{calc, shapes, backgrounds,arrows,positioning,plotmarks}
\tikzset{>=stealth',
  head/.style = {fill = white, text=black},
  plaque/.style = {draw, rectangle, minimum size = 10mm}, 
  pil/.style={->,thick},
  junct/.style = {draw,circle,inner sep=0.5pt,outer sep=0pt, fill=black}
  }

\newcommand{\Z}{\mathbb{Z}}
\newcommand{\Q}{\mathbb{Q}}

\newcommand{\R}{\mathbb{R}}
\newcommand{\C}{\mathbb{C}}

\newcommand{\bP}{\mathbb{P}}
\newcommand{\bA}{\mathbb{A}}

\newcommand{\cO}{\mathcal{O}}
\newcommand{\cB}{\mathcal{B}}
\newcommand{\cC}{\mathcal{C}}

\newcommand{\cI}{\mathcal{I}}
\newcommand{\cP}{\mathcal{P}}
\newcommand{\cR}{\mathcal{R}}
\newcommand{\cS}{\mathcal{S}}
\newcommand{\cT}{\mathcal{T}}

\newcommand{\y}{\mathbf{y}}

\newcommand{\MK}{\overline{M}}

\DeclareMathOperator{\ch}{ch}

\DeclareMathOperator{\QSym}{QSym}
\DeclareMathOperator{\Sym}{Sym}
\DeclareMathOperator{\NSym}{NSym}

\DeclareMathOperator{\Hom}{Hom}

\DeclareMathOperator{\im}{im}

\DeclareMathOperator{\st}{st}
\DeclareMathOperator{\des}{sst}
\DeclareMathOperator{\barred}{{\sf Bar}}

\newcommand{\mQSym}{\mathfrak{m}\!\QSym}
\newcommand{\sQSym}{{\mathcal QSym}}
\newcommand{\mQSymQ}{\mQSym\widehat{\otimes}_\Z\,\Q}
\newcommand{\flags}{\mathsf{Flags}}

\newif\ifhascomments \hascommentstrue
\ifhascomments
  \newcommand{\matt}[1]{{\color{red}[[\ensuremath{\spadesuit\spadesuit\spadesuit} #1]]}}
  \newcommand{\oliver}[1]{{\color{green}[[\ensuremath{\clubsuit\clubsuit\clubsuit} #1]]}}
\else
  \newcommand{\oliver}[1]{}
  \newcommand{\matt}[1]{}
\fi

\AtBeginDocument{%
   \def\MR#1{}
}

\begin{document}

\title{Quasisymmetric Schubert calculus}

\date{\today}

\makeatletter
\@namedef{subjclassname@2020}{%
  \textup{2020} Mathematics Subject Classification}
\makeatother


\author{Oliver Pechenik}
\thanks{OP was partially supported by a Discovery Grant (RGPIN-2021-02391) and Launch Supplement (DGECR-2021-00010) from the Natural Sciences and Engineering Research Council of Canada.}
\address[OP]{Department of Combinatorics \& Optimization, University of Waterloo, Waterloo ON N2L 3G1, Canada}
\email{oliver.pechenik@uwaterloo.ca}

\author{Matthew~Satriano}
\thanks{MS was partially supported by a Discovery Grant from the
  Natural Sciences and Engineering Research Council of Canada and a Mathematics Faculty Research Chair from the University of Waterloo.}
\address[MS]{Department of Pure Mathematics, University
  of Waterloo, Waterloo ON N2L 3G1, Canada}
\email{msatrian@uwaterloo.ca}

\begin{abstract}
The ring of symmetric functions occupies a central place in algebraic combinatorics, with a particularly notable role in the Schubert calculus of Grassmannians, where the standard cell decompositions of Grassmannians yield the celebrated family of Schur functions and the ring structure of cohomology is governed by Littlewood--Richardson rules. The past 50 years have seen an analogous development of \emph{quasisymmetric function theory}, with applications to enumerative combinatorics, Hopf algebras, graph theory, representation theory, and other areas. Despite such successes, this theory has lacked a quasisymmetric analogue of Schubert calculus. In particular, there has been much interest, since work of Lam and Pylyavskyy (2007), in developing ``$K$-theoretic'' analogues of quasisymmetric function theory, for which a major obstacle has been the lack of any topological interpretations.

Here, building on work of Baker and Richter (2008), we develop a quasisymmetric Schubert calculus, applying the philosophy of Schubert calculus to the loop space $\Omega(\Sigma(\C\bP^\infty))$ through the homotopy model given by James reduced product $J(\C\bP^\infty)$.  We describe a canonical Schubert cell decomposition of $J(\C\bP^\infty)$, yielding a canonical basis of its cohomology, which we explicitly identify with monomial quasisymmetric functions. Our constructions apply equally to James reduced products of generalized flag varieties $G/P$, and we show how Littlewood--Richardson rules for any $G/P$ lift to $H^*(J(G/P))$, giving a vast extension of quasisymmetric function theory.

If $J(\C\bP^\infty)$ carried the structure of a normal projective algebraic variety, the structure sheaves of the cell closures would yield a ``cellular $K$-theory'' Schubert basis. We show this is impossible. Nonetheless, we introduce and study a more subtle $K$-theory Schubert basis for $K(J(\C\bP^\infty))\,\widehat{\otimes}_\Z\,\Q$. We characterize this $K$-theory ring and develop quasisymmetric representatives with an explicit combinatorial description.
\end{abstract}

\maketitle

\numberwithin{theorem}{section}
\numberwithin{lemma}{section}
\numberwithin{corollary}{section}
\numberwithin{proposition}{section}
\numberwithin{conjecture}{section}
\numberwithin{question}{section}
\numberwithin{remark}{section}
\numberwithin{definition}{section}
\numberwithin{example}{section}
\numberwithin{notation}{section}
\numberwithin{equation}{section}

\section{Introduction}
\label{sec:intro}

Traditionally, \emph{Schubert calculus} studies the cohomology rings of \newword{generalized flag varieties}, quotients $G/P$ of a Kac--Moody Lie group $G$ by a parabolic subgroup $P$. The key feature facilitating combinatorial analysis of $H^*(G/P)$ is the stratification of $G/P$ by \emph{Schubert varieties}, giving it the structure of a CW-complex with only even-dimensional cells. Such a cell decomposition yields a distinguished linear basis of $H^*(G/P)$, which one models through combinatorial tools, e.g.\ \emph{Schubert polynomials}. This approach yields a very precise understanding of $H^*(G/P)$. For example, when $G/P$ is a complex Grassmannian, its cohomology is governed by the combinatorics of \emph{Schur functions} and the structure coefficients of $H^*(G/P)$ are the \emph{Littlewood-Richardson coefficients}, computed in an explicit positive combinatorial fashion by the 
theory of \emph{Young tableaux} \cite{Littlewood.Richardson} or \emph{puzzles} \cite{Knutson.Tao.Woodward}.

Schur functions are examples of \emph{symmetric functions}; indeed, they are the most celebrated basis of the ring $\Sym$ of symmetric functions. Partly because of this connection, the study of $\Sym$ and its analogues is now one of the central subjects of algebraic combinatorics. The past 50 years have seen a remarkable flowering of \emph{quasisymmetric function theory}, since the introduction \cite{Stanley, Gessel} of the ring $\QSym$ of \emph{quasisymmetric functions} as a tool in enumerative combinatorics. Now, deep connections are known between the combinatorics of $\QSym$ and such diverse additional areas as Hopf algebras (e.g., \cite{Malvenuto.Reutenauer,Ehrenborg,Hoffman,Aguiar.Bergeron.Sottile}), graph theory \cite{Shareshian.Wachs}, probability  \cite{Stanley:riffle,Hersh.Hsiao}, time series \cite{Diehl.EbrahimiFard.Tapia}, Peterson varieties \cite{Nadeau.Tewari:divided}, Macdonald theory \cite{Corteel.Haglund.Mandelshtam.Mason.Williams}, and representation theory (e.g., \cite{Duchamp.Krob.Leclerc.Thibon, Tewari.vanWilligenburg, Searles:extended, Choi.Kim.Nam.Oh}). For additional background on quasisymmetric functions, see \cite{Mason} or the textbook \cite{Luoto.Mykytiuk.vanWilligenburg}.

In this paper, we apply the philosophy of Schubert calculus to the loop spaces $\Omega(\Sigma(G/P))$. (Here, $\Sigma$ and $\Omega$ denote the pointed suspension and loop space functors, where we treat the unique $0$-dimensional Schubert cell of $G/P$ as the basepoint.) Our work is inspired by a paper of A.~Baker and B.~Richter \cite{Baker.Richter} identifying $H^*(\Omega\Sigma \mathbb{CP}^\infty)$ with the ring $\QSym$. Baker and Richter use their topological approach to establish algebraic properties of $\QSym$; however, in contrast to our work, they do not identify a geometrically-natural basis, preventing them from extending the analogy with classical Schubert calculus.  Moreover, we also establish a geometrically-motivated basis from $K$-theory, which has been desired since 2007 work of T.~Lam and P.~Pylyavskyy \cite{Lam.Pylyavskyy}.

\subsection{James reduced products and quasisymmetric functions}\label{sec:J_and_QSym}
We first recall the basic notions of quasisymmetric function theory.
A \newword{composition} $\alpha = (\alpha_1, \dots, \alpha_k)$ is a sequence of positive integers; we write $\ell(\alpha) = k$ for the \newword{length} of $\alpha$.
Consider the power series ring $A = \mathbb{Z} \llbracket x_1, x_2, \dots \rrbracket$ in countably-many variables. A power series $f \in A$ is \newword{quasisymmetric} if, for each composition $(\alpha_1, \dots, \alpha_k)$ and each increasing integer sequence $1 \leq j_1 < j_2 < \dots < j_k$, the coefficients of 
\[
x_1^{\alpha_1} x_2^{\alpha_2} \cdots x_k^{\alpha_k} \quad \text{and} \quad x_{j_1}^{\alpha_1} x_{j_2}^{\alpha_2} \cdots x_{j_k}^{\alpha_k}
\]
in $f$ are equal. Following \cite{Lam.Pylyavskyy}, we write $\mQSym \subset A$ for the ring of quasisymmetric power series and let $\QSym$ denote the subring of quasisymmetric power series of bounded degree. Writing $\QSym_d$ for the homogeneous degree-$d$ part, we have
\[
\QSym=\bigoplus_{d\geq0}\QSym_d
\qquad\text{and}\qquad
\mQSym=\prod_{d\geq0}\QSym_d.
\]
Thus, $\mQSym$ is the degree-completion of $\QSym$.
For each composition $\alpha$, the \newword{monomial quasisymmetric function} $M_\alpha \in \QSym$ is the smallest power series containing the monomial $x_1^{\alpha_1} x_2^{\alpha_2} \cdots x_k^{\alpha_k}$; explicitly,
\[
M_\alpha = \sum_{1 \leq j_1 < j_2 < \dots < j_k} x_{j_1}^{\alpha_1} x_{j_2}^{\alpha_2} \cdots x_{j_k}^{\alpha_k}.
\]
While symmetric functions are invariant under arbitrary finite permutations of the variables, quasisymmetric functions preserve the relative order of the exponents. Thus, compositions, rather than partitions, index the monomial basis of $\QSym$. For example, writing $m_{(2,1)}$ for the monomial symmetric function indexed by the partition $(2,1)$, we have
\[
m_{(2,1)}=M_{(2,1)}+M_{(1,2)}.
\]
The monomial quasisymmetric functions are one of the most classical bases of $\QSym$ and are fundamental to its combinatorial theory (cf.\ \cite{Luoto.Mykytiuk.vanWilligenburg, Mason}).

Our approach to studying $\Omega\Sigma X$ (also followed by \cite{Baker.Richter}) is through the homotopy model given by the \newword{James reduced product} $J(X)$.
For a topological space $X$ with basepoint $e$, its James reduced product \cite{James} (see \cite[VII.2]{Whitehead} for a textbook treatment) is
\[
J(X)=\left(\coprod_{n\geq0}X^n\right)/\sim
\]
where $(x_1,\dots,x_n)\sim(x_1,\dots,x_{i-1},x_{i+1},\dots,x_n)$ if $x_i=e$.  One can think of $J(X)$ as a topologicalization of the free monoid on the points of $X$. The space $J(X)$ is homotopy equivalent to $\Omega\Sigma X$, so we may freely study $H^*(J(X))$ in place of $H^*(\Omega\Sigma X)$.

Let
\[
J_n(X)=X^n/\sim,
\]
where $(x_1,\dots,x_{i-1},e,x_{i+1},\dots,x_n)\sim(x_1,\dots,x_{i-1},x_{i+1},\dots,x_n,e)$.
Then
\[
X=J_1(X)\subseteq J_2(X)\subseteq\dots
\]
and
\[
J(X)=\bigcup_{n\geq1}J_n(X).
\]
Furthermore, if $X$ has a CW-structure with $e$ being a $0$-cell, then the quotient map
\[
q_n\colon X^n\to J_n(X)
\]
endows $J_n(X)$ with the structure of a CW-complex, thereby making $J(X)$ a CW-complex as well. 

We find that $J(\mathbb{CP}^\infty)$ behaves in many ways like a generalized flag variety. 
In Schubert calculus, one has a basis of $H^*(G/P)$ given by the cellular cohomology classes of the Schubert varieties of $G/P$. The cells of $J(\mathbb{CP}^\infty)$ constructed above are analogous to Schubert varieties and their classes are analogous to Schubert classes. In combinatorial Schubert calculus, one identifies Schubert classes with, for example, Schur functions or Schubert polynomials; for us, the monomial quasisymmetric functions $M_\alpha$ play an analogous role.

However, unlike generalized flag varieties, which are naturally smooth quasiprojective (ind)-varieties, $J(\mathbb{CP}^\infty)$ cannot be so realized. This fact, discussed in \autoref{sec:not_normal}, causes some interesting complications when we study the $K$-theory of $J(\mathbb{CP}^\infty)$.

\begin{remark}
	Another appearance of $\Sym$ in Schubert calculus is in the isomorphism 
	\[
	H^*(\flags_n) \cong \mathbb{Z}[x_1, \dots, x_n] / \Sym_n^+,
	\]
	 where $\flags_n$ denotes the \emph{complete flag variety} (see \autoref{ex:flags}) and $\Sym_n^+$ the ideal generated by the symmetric polynomials without constant term. Replacing $\Sym$ here with $\QSym$, one obtains a very interesting quotient ring, first studied in \cite{Aval.Bergeron,Aval.Bergeron.Bergeron}, with significant relations to \emph{permutahedral} and \emph{Peterson varieties} \cite{Nadeau.Tewari:divided,Nadeau.Tewari:forest,Nadeau.Tewari:permutahedral} and to the \emph{Temperley--Lieb algebra} \cite{Bergeron.Gagnon}. We do not yet know how to relate this work to our own.
\end{remark}

\subsection{Cohomology}
Our main cohomological theorem is as follows.

\begin{theorem}
\label{thm:main-H}
The James reduced product $J\C\bP^\infty$ is a CW-complex whose cells $\{e_\alpha\}_\alpha$ are indexed by compositions $\alpha$.
Treating the variables in $\QSym$ as having degree $2$, we have
\[
H^*(J(\C\bP^\infty); \Z) \cong \QSym
\]
 as graded rings.  This isomorphism identifies the monomial quasisymmetric function $M_\alpha$ with the cellular cohomology class of $e_\alpha$.
\end{theorem}

Combining \autoref{thm:main-H} with Hazewinkel's \cite{Hazewinkel} Littlewood--Richardson rule for the multiplication of monomial quasisymmetric functions $M_\alpha$, we show (\autoref{sec:flag}) how one may easily perform very concrete, geometrically-motivated calculations in $H^*(J(\C\bP^\infty)) \cong \QSym$, analogous to the use of classical Littlewood--Richardson rules for Schur functions in explicating the structure of $\Sym$ and the cohomology rings of Grassmannians. In \autoref{sec:flag}, we further extend \autoref{thm:main-H} and the accompanying multiplication rules from $J\C\bP^\infty$ to the James reduced products of other flag varieties, such as Grassmannians.

The isomorphism of $H^*(J(\C\bP^\infty))$ with $\QSym$ is due originally to \cite{Baker.Richter}; our alternative proof is arguably more explicit and avoids consideration of the ring $\NSym$ of \emph{noncommutative symmetric functions}. Our identification of geometric underpinnings for a basis for $H^*(J(\C\bP^\infty))$ and its associated combinatorics is new.

\begin{remark}
    Oesinghaus \cite{Oesinghaus}, building on the \emph{expanded pairs} of \cite{Abramovich.Cadman.Fantechi.Wise}, identifies $\QSym$ with the Chow ring of an Artin stack, but in a way that does not manifest a geometrically-natural basis. From this perspective, Bae--Schmitt \cite{Bae.Schmitt} later identified the monomial quasisymmetric basis with certain tautological classes in the Chow ring. It would be interesting to determine an explicit relation between these constructions and our own, which seem similar in flavor due to the fact that their construction has an atlas given by an $n$-fold product stack $[\bA^1/\mathbb{G}_m]^n$. Our construction, however, applies to more general spaces and extends to $K$-theory. In addition, our space $J(\C\bP^\infty)$ can be given the structure of a scheme \cite{Pechenik.Satriano:double}, whereas the stack of expanded pairs is a genuine Artin stack.
    
\end{remark}

\subsection{\texorpdfstring{$K$}{K}-theory}
There has been much interest (e.g., \cite{Patrias, Monical, Pechenik.Searles, Lewis.Marberg,Monical.Pechenik.Searles}), since pioneering work of T.~Lam and P.~Pylyavskyy \cite{Lam.Pylyavskyy}, in developing ``$K$-theoretic'' analogues of quasisymmetric function theory. A major obstacle has been the lack of any topological interpretations of quasisymmetric bases. With \autoref{thm:main-H} in hand, we therefore 
 turn to a $K$-theoretic analogue. 

In Schubert calculus, one traditionally (see, for example, \cite{Fulton.Lascoux,Buch,Pechenik.Yong}) associates a $K$-theory class to a Schubert variety by considering the structure sheaf of the variety (although other related choices are possible, e.g.\ \cite{Thomas.Yong,Knutson:Osaka}). Here, the structure sheaf naturally lives in the Grothendieck group $K_0(G/P)$ of coherent sheaves, but one can transfer it to the ring $K^0(G/P)$ by taking a resolution by locally-free sheaves, since $G/P$ is smooth and projective. Such an approach for $J(\C\bP^\infty)$ is not viable since, as observed in \autoref{sec:not_normal}, $J(\C\bP^\infty)$ cannot be realized as a normal quasi-projective ind-variety. (In our follow-up paper \cite{Pechenik.Satriano:double}, we give a realization of $J(\C\bP^\infty)$ as a non-normal ind-variety, but this does not resolve the current difficulty.) 

Nonetheless, we introduce and study a more subtle $K$-theory Schubert basis $[e_\alpha]$ for $K(J(\C\bP^\infty))\,\widehat{\otimes}_\Z\,\Q$. The construction of our new ``cellular $K$-theory basis'' involves delicate combinatorics and appears to rely on special features of $\C\bP^\infty$. It would be interesting to explore whether analogous bases exist for related spaces. (We need the tensor with $\mathbb{Q}$ for technical reasons; it might additionally be possible to drop the tensor product by developing further combinatorics.) 

On the quasisymmetric function theory side, we introduce power series $\MK_\alpha \in \mQSym$, indexed by compositions and deforming the monomial quasisymmetric functions $M_\alpha$. More precisely, $\MK_\alpha$ is $M_\alpha$ plus higher-degree terms. These new power series, which we call \newword{quasisymmetric monomial glides}, are given by an explicit combinatorial formula (see \autoref{sec:glides-via-Mobius}) and are closely related to the \emph{monomial} and \emph{fundamental slides} of \cite{Assaf.Searles}, as well as to the \emph{fundamental glides} of \cite{Pechenik.Searles} (see \cite{Searles,Pechenik.Searles:survey,Monical.Pechenik.Searles} for related discussion). In \autoref{thm:K-classesAgreewithGlidePolys}, we show that the set of quasisymmetric monomial glides
$\MK_\alpha$ forms a \emph{Schauder basis} for the ring of quasisymmetric power series with rational coefficients; that is, every quasisymmetric power series may be uniquely written as a sum $\sum_\alpha c_\alpha \MK_\alpha$ with $c_\alpha \in \Q$ and converging in the usual topology of power series. Moreover, we establish the following $K$-theoretic analogue of \autoref{thm:main-H}, tying quasisymmetric monomial glides to the $K$-theoretic Schubert calculus of $J(\C\bP^\infty)$.

\begin{theorem}
\label{thm:main-K}
We have an isomorphism
\[
K(J(\C\bP^\infty))\,\widehat{\otimes}_\Z\,\Q\ \cong\ \mQSymQ.
\] This isomorphism identifies the $K$-class $[e_\alpha]$ with the quasisymmetric monomial glide $\MK_\alpha$. 
\end{theorem}

\autoref{thm:main-K} follows from the more precise version established in \autoref{thm:K-classesAgreewithGlidePolys}.

Note that the ring $\mQSymQ$ is nothing more than quasisymmetric power series with rational coefficients.  
The ring $\mQSym$ has been previously studied, for example by \cite{Lam.Pylyavskyy,Patrias,Lewis.Marberg}, in relation to combinatorial constructions considered to be ``$K$-theoretic''; to our knowledge, however, $\mQSym$ has never before been explicitly identified with a $K$-theory ring, nor have nonsymmetric elements of $\mQSym$ been previously identified with $K$-theory classes.

A slightly different ``$K$-theoretic'' analogue of monomial quasisymmetric functions was proposed briefly in \cite[Remark~5.14]{Lam.Pylyavskyy} by T.~Lam and P.~Pylyavskyy; these were inspired by certain other inhomogeneous power series related to the $K$-theory of Grassmannians; however, no explicit connection with $K$-theoretic geometry was obtained. We had initially imagined that the isomorphism of \autoref{thm:main-K} would identify $[e_\alpha]$ with those Lam--Pylyavskyy \emph{multimonomial quasisymmetric functions}, and it is interesting to see that this is not the case. See \autoref{rem:LP} for a more precise comparison.

In addition to giving geometric insight into the combinatorics of $\QSym$ and $\mQSym$, we hope our new combinatorial models for the cohomology and $K$-theory of the loop space $\Omega\Sigma\C\bP^\infty$ give some insight towards the geometric content of \emph{elliptic cohomology}. Elliptic cohomology is a complex oriented cohomology theory that arises from formal algebraic considerations. A geometric perspective on what is measured by elliptic cohomology would be expected to give many important insights, but is impeded by the difficulty of performing explicit calculations even for simple spaces. See, for example, \cite{Lenart.Zainoulline,Rimanyi.Weber,Kumar.Rimanyi.Weber} for work developing the beginnings of a theory of elliptic Schubert calculus and \cite{BerwickEvans.Tripathy} for geometric models of elliptic cohomology. 
Since the elliptic cohomology of a space $X$ is approximated by the $K$-theory of the loop space $\Omega X$, it is useful to have spaces with a concrete understanding of $K^*(\Omega X)$. Our new combinatorics provides such an understanding for the case $X = \Sigma \C\bP^\infty$. 

To connect with the topology of \autoref{thm:main-K}, we find it useful to define the quasisymmetric monomial glides via a cancellative formula involving the \emph{M\"obius function} on a certain poset. However, to calculate explicitly with quasisymmetric monomial glides it is preferable to have an explicit non-cancellative combinatorial formula. Indeed, we end up needing such a formula to see that quasisymmetric monomial glides are in fact quasisymmetric! In \autoref{sec:glides-via-Mobius}, we establish such a formula by constructing a somewhat complicated sign-reversing involution.  Our formula, which may be of independent interest, uses ``sliding'' combinatorics similar to that appearing in \cite{Assaf.Searles, Pechenik.Searles}.  The main combinatorial challenge, compared to many works (e.g., \cite{Kreweras,Sagan.Vatter,McNamara.Sagan}) establishing non-cancellative formulas for M\"obius functions of other posets, is that the posets we are forced to consider are generally not \emph{graded}, which complicates the analysis.

\medskip
\noindent {\bf This paper is organized as follows.} In \autoref{sec:cohomology}, we establish \autoref{thm:main-H}, matching the cellular basis of $H^*(J(\C\bP^\infty))$ with the combinatorics of monomial quasisymmetric functions. In \autoref{sec:glides-via-Mobius}, we introduce quasisymmetric monomial glides and prove a combinatorial formula for them. In \autoref{sec:K-thy-james-reduced-prod}, we construct a ``cellular'' $K$-theory basis of $K(J(\C\bP^\infty))\,\widehat{\otimes}_\Z\,\Q$ and connect it to the combinatorics of $\mQSym$ and quasisymmetric monomial glides, establishing \autoref{thm:main-K}. The brief \autoref{sec:not_normal} explains the geometric obstacles to simplifying the construction of \autoref{sec:K-thy-james-reduced-prod}. Finally, in \autoref{sec:flag}, we sketch the extension of \autoref{thm:main-H} to James reduced products of a general class of CW complexes, with special emphasis on finite-dimensional generalized flag varieties and the classifying spaces $BU(k)$. In particular, \autoref{thm:BUk} gives an explicit positive combinatorial rule for the cellular structure coefficients of $H^*(J(BU(k)))$.

\section{Quasisymmetric functions from the cells of the James reduced product}\label{sec:cohomology}

This section has two goals. We first recall the monomial basis of $\QSym$ and Hazewinkel's overlapping-shuffle rule for its multiplication, which will later describe cup products in a cellular basis. We then describe the cells of $J(\C\bP^\infty)$ and show that pullback along the quotient maps $q_n$ sends their cohomology classes to monomial quasisymmetric functions. Passing to the inverse limit proves \autoref{thm:main-H}.

We begin by fixing notation that will be used throughout the paper. A \newword{weak composition} is a finite sequence of nonnegative integers; recall from the introduction that a \emph{composition} is a finite sequence of positive integers. Given a weak composition $\alpha$, its \newword{positive part} $\alpha^+$ is the composition obtained from $\alpha$ by deleting all $0$ terms. For example, the positive part of $(2,0,4,0,0,2)$ is $(2,4,2)$. We write $|\alpha|$ for the sum of the parts of $\alpha$. As a shorthand, for positive integers $n \leq n'$, we write $[n] = \{ 1, 2, \dots, n\}$ and $[n,n'] = \{n, n+1, \dots, n'\}$.

Recall $A = \Z\llbracket x_1, x_2, \dots\rrbracket$. For any weak composition $\alpha$ of length $\ell$, we have a monomial $\prod_{i=1}^\ell x_i^{\alpha_i}$; we write $[x^\alpha]f$ for the coefficient of this monomial in the power series $f$. We topologize $A$ so that a sequence of power series $f_1, f_2, \dots$ converges to $f \in A$ if and only if, for each weak composition $\alpha$, there is some $m$ such that $[x^\alpha]f_M = [x^\alpha]f$ for all $M \geq m$. A power series $f \in A$ is quasisymmetric if, whenever $\alpha^+ = \beta^+$, we have
\[
[x^\alpha] f = [x^\beta] f.
\]
We say that $f\in A$ has \newword{bounded degree} if there exists $D\geq 0$ such that $[x^\alpha]f=0$ whenever $|\alpha|>D$.
We write $\mQSym$ for the set of quasisymmetric power series and $\QSym$ for the subset of bounded degree; elements of the latter are called \newword{quasisymmetric functions}. It is straightforward to check that both sets are subrings of $A$. They may be fruitfully understood as the invariants of a non-algebraic action of the symmetric group \cite{Hivert}; however, we shall not need this perspective here. The ring $\QSym$ has the structure of a \emph{combinatorial Hopf algebra} in the sense of \cite{Aguiar.Bergeron.Sottile} and indeed is the final object in the category of such Hopf algebras. It would be interesting to explore the relations between the geometry of $J(\C\bP^\infty)$, the coproduct on $\QSym$, and the dual Hopf algebra of \emph{noncommutative symmetric functions}, for which there are rich combinatorics (see, e.g., \cite{Hicks.McCloskey}). Here, we focus on the vector space and ring structure of $\QSym$ (and $\mQSym$), as well as the associated combinatorics.

Like the more classical ring of symmetric functions, $\QSym$ has a number of bases of combinatorial interest. (See \autoref{sec:J_classifying} for background on symmetric functions.) Most basic of these is the basis of monomial quasisymmetric functions as defined in \autoref{sec:J_and_QSym}, which is analogous to the monomial basis of symmetric functions. There does not appear to be an entirely satisfactory quasisymmetric analogue of the famous \emph{Schur basis} of symmetric functions. Rather, there are many distinct quasisymmetric analogues, each sharing some but not all of the desirable properties of the Schur basis; among these are the \emph{fundamental} basis \cite{Gessel}, the \emph{quasisymmetric Schur} basis \cite{Haglund.Luoto.Mason.vanWilligenburg}, the \emph{dual immaculate} basis \cite{Berg.Bergeron.Saliola.Serrano.Zabrocki}, the \emph{row-strict dual immaculate} basis \cite{Niese.Sundaram.vanWilligenburg.Vega.Wang}, the \emph{extended Schur} basis \cite{Assaf.Searles:kohnertPoly}, etc. See \cite{Luoto.Mykytiuk.vanWilligenburg,Mason,Pechenik.Searles:survey} for further discussion. A possibly surprising consequence of the current paper is that the monomial basis of $\QSym$ can also be understood as an analogue of the Schur basis, in that it describes an analogous sort of Schubert calculus.

Following \cite{Hazewinkel}, we now combinatorially describe the product structure of $\QSym$ with respect to the basis of monomial quasisymmetric functions.
Let $(m,n)$ be a pair of positive integers. Then, an \newword{overlapping shuffle} of $(m,n)$ is a surjection 
\[
\tau : [m+n] \to [k]
\] 
for some $\max\{m,n\} \leq k \leq m+n$ such that 
\begin{itemize}
	\item $i<j \leq m \Rightarrow \tau(i) < \tau(j)$ and
	\item $m<i<j \Rightarrow \tau(i) < \tau(j)$.
\end{itemize} 
That is to say, $\tau$ is separately strictly order preserving on the first $m$ elements of $[m+n]$ and on the last $n$ elements, but has no requirements on the relative behaviour of elements from opposite ends of the interval $[m+n]$.

Let $\alpha$ be a composition of length $m$ and $\beta$ a composition of length $n$. Then, for $\tau$ an overlapping shuffle of $(m,n)$, we define a composition $\gamma^\tau$ by
\begin{equation}\label{eq:gamma}
	\gamma_i^\tau = \sum_{\tau(j) = i} (\alpha^\frown \beta)_j,
\end{equation}
where $\alpha^\frown \beta$ denotes concatenation of strings. (Note that here the summation has at most $2$ summands.) The \newword{overlapping shuffle product} of $\alpha$ and $\beta$ is the formal sum
\begin{equation}\label{eq:shuffle}
	\alpha \shuffle_o \beta  = \sum_{\tau} \gamma^\tau
\end{equation}
over overlapping shuffles $\tau$ of $(m,n)$. (Here, the binary operator is a modified Cyrillic letter ``Sha'' for ``shuffle,'' with a subscript ``o'' for ``overlapping'' to distinguish from the more common \emph{shuffle product} of S.~Eilenberg and S.~Mac~Lane \cite{Eilenberg.MacLane}.) Note that distinct overlapping shuffles $\tau, \tau'$ can yield the same composition $\gamma^\tau = \gamma^{\tau'}$, so the formal sum \eqref{eq:shuffle} can have nontrivial coefficients.

Intuitively, the overlapping shuffle interleaves the parts of two compositions, with the additional option of merging one part from each composition by replacing them with their sum.

\begin{example}\label{ex:overlapping_shuffle}
	The overlapping shuffle product of the compositions $\alpha =(3)$ and $\beta = (1,3)$ is
	\[
	(3) \shuffle_o (1,3) = (3,1,3) + 2 \cdot (1,3,3) + (4,3) + (1,6).
	\]
	The composition $(1,3,3)$ appears once from an overlapping shuffle $\tau$ with $\tau(1) = 2$ and once from an overlapping shuffle $\tau'$ with $\tau'(1) = 3$.
\end{example}

The following is a monomial quasisymmetric function analogue of the \emph{Littlewood--Richardson rule} (see \autoref{prop:LR}) for Schur functions. 

\begin{proposition}[\cite{Hazewinkel}]\label{prop:M-LR}
For compositions $\alpha, \beta$, the corresponding monomial quasisymmetric functions multiply as
\[
M_\alpha \cdot M_\beta = \sum_\gamma c_{\alpha, \beta}^\gamma M_\gamma,
\]	
where $c_{\alpha, \beta}^\gamma$ denotes the multiplicity of the composition $\gamma$ in the overlapping shuffle product $\alpha \shuffle_o \beta$.
\end{proposition}
\begin{proof}[Proof (sketch)]
	Consider the monomial expansions of $M_\alpha$ and $M_\beta$ from their definitions. Now multiply, distributing term by term. Each monomial appearing from this product corresponds to a $\gamma$ from the sum on the right. It is not hard to see that the multiplicities are also correct.
\end{proof}

We will extend this result in \autoref{sec:J_classifying}.

\begin{example}
	From \autoref{ex:overlapping_shuffle} and \autoref{prop:M-LR}, we have that
	\[
	M_{(3)} \cdot M_{(1,3)} = M_{(3,1,3)} + 2M_{(1,3,3)} + M_{(4,3)} + M_{(1,6)}.
	\]
	The reader may enjoy checking this computation directly after restricting to a small number of variables.
\end{example}

To set conventions for projective spaces, recall (see, e.g., \cite[Example 0.6]{hatcher-at}) that $\C\bP^n$ is defined as the quotient of $\C^{n+1}\setminus0$ by complex scalar multiplication, or equivalently as the quotient of the unit sphere in $\C^{n+1}$ by multiplication by complex numbers of norm $1$. It has cells $e_0,e_2,\dots,e_{2n}$ with $\dim_\R e_i=i$ where the cell $e_{2i}$ is attached via the quotient map of the boundary sphere $S^{2i-1}\to\C\bP^{i-1}$. We have inclusions $\C^{n+1}\setminus0 \subset \C^{n+2}\setminus0$ sending $(x_0,\dots,x_n)$ to $(x_0,\dots,x_n,0)$, which yields inclusions $\C\bP^n\subset\C\bP^{n+1}$. Hence, $\C\bP^\infty \coloneqq \bigcup_n\C\bP^n$ has cells $e_0,e_2,e_4,\dots$ with $\dim_\R e_i = i$. We will always use this CW-structure on $\C\bP^n$ and $\C\bP^\infty$.

\begin{lemma}\label{l:cellstructureJn}
    The cells of $J_n\C\bP^\infty$ are indexed by compositions $\alpha=(\alpha_1,\dots,\alpha_k)$ with length $k\leq n$. The cells of $J\C\bP^\infty$ are indexed by the set of all compositions. We have
    \[
        H_*(J(\C\bP^\infty); \Z)\cong\bigoplus_\alpha \Z e_\alpha\quad\textrm{and}\quad H^*(J(\C\bP^\infty); \Z)\cong\bigoplus_\alpha \Z x_\alpha,
    \]
    where $x_\alpha$ denotes the function dual to the cell $e_\alpha$ indexed by the composition $\alpha$. Lastly, \begin{equation}\label{eqn:H-is-inverse-limit-James-reduced-product}
        H^*(J(\C\bP^\infty); \Z)=\lim_{n \to \infty} H^*(J_n(\C\bP^\infty); \Z),
    \end{equation}
    where the right-hand side is the inverse limit induced by the fact that $J(\C\bP^\infty)$ is the colimit of the $J_n(\C\bP^\infty)$.
\end{lemma}
\begin{proof}
By definition, the CW complex structure on $J_n\C\bP^\infty$ is induced from the cellular quotient map
\[
q_n\colon (\C\bP^\infty)^n\to J_n\C\bP^\infty.
\]
By construction, $q_n$ identifies the cells $e_{2b_1}\times\dots\times e_{2b_n}$ and $e_{2c_1}\times\dots\times e_{2c_n}$ of $(\C\bP^\infty)^n$ if and only if $(b_1,\dots,b_n)^+ = (c_1,\dots,c_n)^+$. Thus, the cells of $J_n\C\bP^\infty$ are indexed by compositions $\alpha=(\alpha_1,\dots,\alpha_k)$ with length $k\leq n$. It follows that the cells of $J\C\bP^\infty$ are indexed by the set of all compositions. For any composition $\alpha$, we let
\[
e_\alpha\subset J\C\bP^\infty
\]
denote the corresponding cell. Since all cells of $J\C\bP^\infty$ are even-dimensional, we see that its integral homology is freely generated by the cells; hence, we will not distinguish between cells and their cellular homology classes:
\[
H_*(J(\C\bP^\infty); \Z)\cong\bigoplus_\alpha \Z e_\alpha.
\]
Similarly, $H^d(J(\C\bP^\infty);\Z)=\Hom(H_d(J(\C\bP^\infty);\Z),\Z)$, so
\[
H^*(J(\C\bP^\infty); \Z)\cong\bigoplus_\alpha \Z x_\alpha,
\]
where $x_\alpha$ denotes the function dual to $e_\alpha$. In particular, we see from this description of the cells of $J_n(\C\bP^\infty)$ and $J(\C\bP^\infty)$ that if 
\[
\iota_n\colon J_n(\C\bP^\infty)\to J(\C\bP^\infty)
\]
denotes the inclusion, then 
\[
\iota_n^*\colon H^{2d}(J(\C\bP^\infty); \Z)\to H^{2d}(J_n(\C\bP^\infty);\Z)
\]
is an isomorphism for all $n\geq d$. As a result
\[
H^*(J(\C\bP^\infty); \Z)=\lim_{n \to \infty} H^*(J_n(\C\bP^\infty); \Z). \qedhere
\]    
\end{proof}

\autoref{l:cellstructureJn} establishes the first sentence of \autoref{thm:main-H}; the remainder is immediately implied by \autoref{james-reduced-cells->monomial-basis} below. Recall that $\QSym$ denotes the graded ring of quasisymmetric functions in variables $x_1,x_2,\dots$, where each $x_i$ has degree $2$.

\begin{theorem}
\label{james-reduced-cells->monomial-basis}
We have an isomorphism of graded rings
\[
H^*(J(\C\bP^\infty); \Z) \cong \QSym
\]
sending $x_\alpha$ to the quasisymmetric monomial function $M_\alpha$.
\end{theorem}
\begin{proof}
First, by the K\"unneth formula,
\[
H^*((\C\bP^\infty)^n; \Z)=H^*(\C\bP^\infty; \Z)^{\otimes n}\xleftarrow[\cong]{\phi_n}\Z[x_1,\dots,x_n];
\]
the isomorphism $\phi_n$ identifies $x_i$ with the dual of the cell $e_0\times\dots\times e_0\times e_2\times e_0\times\dots\times e_0$, where $e_2$ appears in the $i$th position.

We see that if $k\leq n$, then 
\[
q_n^{-1}(e_{(\alpha_1,\dots,\alpha_k)})=\coprod_{\iota} e_\iota,
\]
where 
\begin{itemize}
	\item $\iota\colon[k]\to[n]$ is an order-preserving injection;
	\item $e_\iota=e_{i_1}\times\dots\times e_{i_n}$, where $i_{\iota(j)}=2\alpha_j$ for $j\in[k]$ and $i_r=0$ for $r\notin\im\iota$.
\end{itemize}
 As a result,
\[
q_n^*(x_{(\alpha_1,\dots,\alpha_k)})=\sum_{1\leq i_1<\dots<i_k\leq n}x^{\alpha_1}_{i_1}\dots x^{\alpha_k}_{i_k}=:M_{n,(\alpha_1,\dots,\alpha_k)};
\]
notice that $M_{n,(\alpha_1,\dots,\alpha_k)}$ is the quasisymmetric monomial function $M_{(\alpha_1,\dots,\alpha_k)}$, truncated to the finite set of variables $x_1,\dots,x_n$.

Let
\[
\QSym_n \subset \Z[x_1,\dots,x_n]
\]
be the subring of \newword{quasisymmetric polynomials} in $x_1,\dots,x_n$. Since $H^*(J_n(\C\bP^\infty); \Z)$ is freely generated by all $x_{(\alpha_1,\dots, \alpha_k)}$ with $k\leq n$, we see the image of $q_n^*$ is contained in $\QSym_n$. We therefore obtain a graded ring map
\begin{equation}\label{eqn:mapQSymm}
q_n^*\colon H^*(J_n(\C\bP^\infty); \Z)\longrightarrow\QSym_n.
\end{equation}
This map is surjective since $q_n^*(x_{(\alpha_1,\dots, \alpha_k)})=M_{n,(\alpha_1,\dots, \alpha_k)}$ and the elements $M_{n,(a_1,\dots,a_k)}$ with $k\leq n$ form a $\Z$-basis for $\QSym_n$. Furthermore, the map $q_n^*$ is injective since if $f\in H^*(J_n(\C\bP^\infty); \Z)$, then the coefficient of $x_{(\alpha_1,\dots, \alpha_k)}$ in the expression for $f$ is equal to the coefficient of $x_1^{\alpha_1}\dots x_k^{\alpha_k}$ in the expression for $q_n^*(f)$. Thus, \eqref{eqn:mapQSymm} is an isomorphism of graded rings.

To finish the proof, we note that for all $n$, we have a commutative diagram
\[
\xymatrix{
(\C\bP^\infty)^n\ar[r]^-{\jmath'_n}\ar[d]^-{q_n} & (\C\bP^\infty)^{n+1}\ar[d]^-{q_{n+1}}\\
J_n(\C\bP^\infty)\ar[r]^-{\jmath_n} & J_{n+1}(\C\bP^\infty)
}
\]
where $\jmath'_n$ and $\jmath_n$ map $(p_1,\dots,p_n)$ to $(p_1,\dots,p_n,e_0)$. This induces a commutative diagram
\[
\xymatrix{
H^*(J_{n+1}(\C\bP^\infty);\Z)\ar[r]^-{q_{n+1}^*}_-{\cong}\ar[d]_-{\jmath_n^*} & \QSym_{n+1} \ar@{^{(}->}[r]\ar[d]_-{\pi_n} & \Z[x_1,\dots,x_{n+1}] \ar[d]^-{\pi'_n} & H^*((\C\bP^\infty)^{n+1})\ar[d]^-{(\jmath'_n)^*}\ar[l]^-{\cong}_-{\phi_{n+1}}\\
H^*(J_{n}(\C\bP^\infty);\Z)\ar[r]^-{q_{n}^*}_-{\cong} & \QSym_{n} \ar@{^{(}->}[r] & \Z[x_1,\dots,x_n]  & H^*((\C\bP^\infty)^{n})\ar[l]^-{\cong}_-{\phi_n}
}
\]
of graded rings; the maps $\pi_n$ and $\pi'_n$ kill $x_{n+1}$ and send $x_i$ to $x_i$ for $i\leq n$. Combining this commutative diagram with equation \eqref{eqn:H-is-inverse-limit-James-reduced-product}, we see that
\[
H^*(J(\C\bP^\infty);\Z)=\lim_n H^*(J_n(\C\bP^\infty);\Z)\cong\lim_n \QSym_n=\QSym,
\]
where the limit is taken in the category of graded rings.
\end{proof}

By \autoref{james-reduced-cells->monomial-basis}, \autoref{prop:M-LR} yields a positive combinatorial formula for the structure coefficients of $H^*(J(\C\bP^\infty))$ with respect to the cellular basis $\{x_\alpha \}$. In \autoref{sec:J_classifying}, we extend this result from $H^*(J(\C\bP^\infty)) = H^*(J(BU(1)))$ to other classifying spaces.

%\begin{remark}
%For the reader less familiar with limits in the category of graded rings, we give a few more details about why $\lim_n \QSym_n=\QSym$. The limit in the category of graded rings is constructed by taking the direct sum of the limits of each graded piece. Letting, $\QSym_n^d$ denote the quasisymmetric polynomials of degree $d$, we have then
%\[
%\lim_n \QSym_n=\bigoplus_d \lim_n\QSym_n^d.
%\]
%The limit $\lim_n\QSym_n^d$ is now taken in the category of all rings, and therefore consists of all quasisymmetric power series of degree $d$. As a result,
%\[
%\lim_n \QSym_n=\QSym.
%\]
%\end{remark}

\section{Quasisymmetric monomial glides}
\label{sec:glides-via-Mobius}

In this section, we first define the quasisymmetric monomial glides $\MK_\alpha$ via a \emph{cancellative} formula using a M\"obius function on a certain poset. This characterization will be useful to connect with the geometry in \autoref{sec:K-thy-james-reduced-prod}, where we establish that $\MK_\alpha$ represents the $K$-class of the cell $e_\alpha$ of $J(\C\bP^\infty)$. On the other hand, this characterization is inefficient for other uses, due to the cancellations and the complexity of the poset. Moreover, the quasisymmetry of $\MK_\alpha$ is not easily apparent from this definition. The main task of this section then is to establish a simpler non-cancellative formula for the $\MK_\alpha$, which additionally manifests their quasisymmetry. We will additionally need this quasisymmetry in \autoref{sec:K-thy-james-reduced-prod}.

Throughout this section, let
\begin{equation}\label{eqn:alpha-underbrace-expression}
\alpha=(\underbrace{a_1,\dots,a_1}_{N_1},\underbrace{a_2,\dots,a_2}_{N_2},\dots,\underbrace{a_k,\dots,a_k}_{N_k})
\end{equation}
be a composition where $a_i\neq a_{i+1}$ for $1\leq i<k$. Fix
\[
n\geq N \coloneqq \sum_{i=1}^k N_i.
\]

\subsection{Definitions and basic properties}
We introduce the following objects.

\begin{definition}\label{def:SalphaPalpha}
Let $\mathcal{S}_\alpha$ be the $\binom{n}{N}$-element set of all length-$n$ strings obtained from $\alpha$ by inserting $0$s.
That is to say, \[
\mathcal{S}_\alpha = \{(b_1, \dots, b_n) : b \text{ is a weak composition with $b^+ = \alpha$}\},
\] 
where we are conflating tuples with strings. Let $\cP_\alpha$ be defined as the closure of the set $\mathcal{S}_\alpha$ under componentwise maximum; that is, $\cP_\alpha$ is the set of tuples $(c_1,\dots,c_n)$ for which there exist an integer $m\geq 1$ and elements \[
(b_{11},\dots,b_{1n}),\dots,(b_{m1},\dots,b_{mn})\in\mathcal{S}_\alpha
\]
such that for every $j$, we have $c_j=\max_i b_{ij}$. 
%Let $\cP_\alpha$ be the closure of the set $\mathcal{S}_\alpha$ under componentwise maximum, i.e., $(c_1,\dots,c_n)\in\cP_\alpha$ if there exist $(b_{i1},\dots,b_{in})\in\mathcal{S}_\alpha$ such that $c_j=\max_i b_{ij}$. 
We give $\cP_\alpha$ the structure of a poset by componentwise comparison, so that
\[
(b_1, \dots, b_n) \leq (b_1', \dots, b_n') \text{ if and only if $b_i \leq b_i'$ for all } i.
\]
\end{definition}

\begin{example}\label{ex:Hasse}
	Let $\alpha = (1,3)$ and let $n = 4$. Then \[
	\mathcal{S}_\alpha = \{ 0013,0103,0130,1003,1030,1300 \},
	\] where we drop commas and parentheses in weak compositions for concision. The reader may check that, as a set, 
	\[
	\cP_\alpha = \mathcal{S}_\alpha \cup \{ 0113,1103,1013,1130, 1113,1330,0133,1033,1303,1133,1313,1333\}.
	\]
	We may visually illustrate the poset structure on $\cP_\alpha$ by the \emph{(Hasse) diagram} below:
 
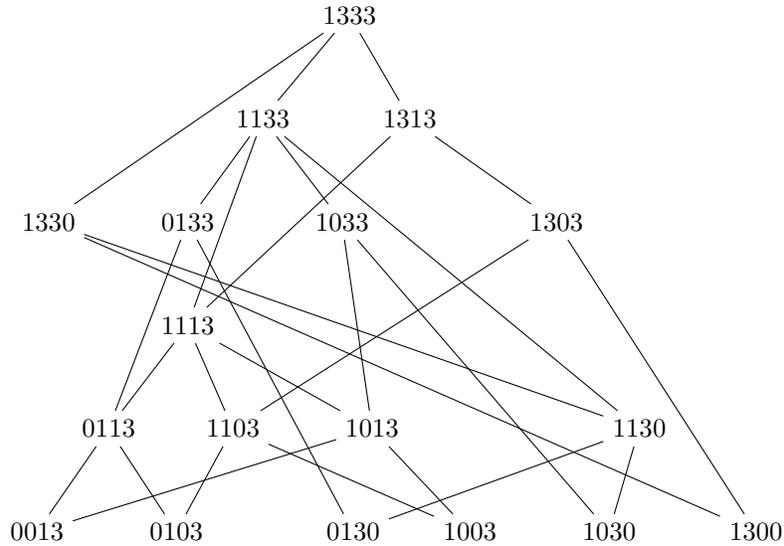
\[
 \begin{tikzpicture}
	\node (0013) at (0,0) {$0013$};
	\node[right= 0.9 of 0013] (0103) {$0103$};
	\node[right= 1.4 of 0103] (0130) {$0130$};
	\node[right= 0.6 of 0130] (1003) {$1003$};
	\node[right= 0.9 of 1003] (1030) {$1030$};
	\node[right= 1.0 of 1030] (1300) {$1300$};
	\node[above right= 0.9 and 0.0 of 0013] (0113) {$0113$}; %first Ktheory
	\node[right= 0.7  of 0113] (1103) {$1103$};
	\node[right= 0.9 of 1103] (1013) {$1013$};
	\node[right = 2.6 of 1013] (1130) {$1130$}; 
	\node[above right = 0.9 and 0.1 of 0113] (1113) {$1113$}; %rank 3
	\node[above left = 0.9 and 0.9 of 1113] (1330) {$1330$}; %rank 4
	\node[right = 0.9 of 1330] (0133) {$0133$}; 
	\node[right = 1.1 of 0133] (1033) {$1033$}; 
	\node[right = 1.9 of 1033] (1303) {$1303$}; 
	\node[above right = 0.9 and 1.9 of 1330] (1133) {$1133$}; % rank 5
	\node[right = 1.0 of 1133] (1313) {$1313$}; 
	\node[above right = 0.9 and 0.2 of 1133] (1333) {$1333$}; %top one
%	\node[below = 3.9  of 22101] (0hat) {$\widehat{0}$}; %bottom one
\draw (0013) to (0113);
\draw (0013) to (1013);
\draw (0103) to (0113);
\draw (0103) to (1103);
\draw (0130) to (1130);
\draw (0130) to (0133);
\draw (1003) to (1103);
\draw (1003) to (1013);
\draw (1030) to (1130);
\draw (1030) to (1033);
\draw (1300) to (1330);
\draw (1300) to (1303);
\draw (0113) to (0133);
\draw (0113) to (1113);
\draw (1103) to (1113);
\draw (1103) to (1303);
\draw (1013) to (1113);
\draw (1013) to (1033);
\draw (1130) to (1330);
\draw (1130) to (1133);
\draw (1113) to (1133);
\draw (1113) to (1313);
\draw (1330) to (1333);
\draw (0133) to (1133);
\draw (1033) to (1133);
\draw (1303) to (1313);
\draw (1133) to (1333);
\draw (1313) to (1333);
%\draw (11111.north) to (21111);
%\draw (21110) to (21111);
%\draw (21101.north) to (21111);
%\draw (21011.north) to (21111);
%\draw (11111.north) to (12111);
%\draw (12110.north) to (12111);
%\draw (12101.north) to (12111);
%\draw (12011) to (12111);
%\draw (02111) to (12111);
%\draw (21110.north) to (22110);
%\draw (12110.north) to (22110);
%\draw (21101.north) to (22101);
%\draw (12101.north) to (22101);
%\draw (21011.north) to (22011);
%\draw (12011.north) to (22011);
%\draw (21111) to (22111);
%\draw (22110) to (22111);
%\draw (22101) to (22111);
%\draw (22011) to (22111);
%\draw (12111) to (22111);
%\draw (0hat) to (20111.south);
%\draw (0hat) to (11111.south);
%\draw (0hat) to (21110.south);
%\draw (0hat) to (12110);
%\draw (0hat) to (21101);
%\draw (0hat) to (12101);
%\draw (0hat) to (21011);
%\draw (0hat) to (12011.south);
%\draw (0hat) to (02111.south);
\end{tikzpicture}.
 \]
	Here, $b \leq b'$ if and only if one can reach $b'$ from $b$ via a sequence of ascending edges.
\end{example}

Let $\widehat{\cP}_\alpha$ be the poset obtained from $\cP_\alpha$ by adjoining a minimum element $\widehat{0}$ satisfying $\widehat{0} < p$ for all $p \in \cP_\alpha$. 

For a poset $\mathcal{P}$ and elements $p,q \in \mathcal{P}$, a \newword{lower bound} of $p$ and $q$ is any element $r \in \mathcal{P}$ such that $r \leq p$ and $r \leq q$. Similarly, an \newword{upper bound} is any element $s \in \mathcal{P}$ with $s \geq p$ and $s \geq q$. A 
\newword{meet} of $p$ and $q$ is an element $p \wedge q \in \mathcal{P}$ such that 
\begin{itemize}
	\item $p \wedge q$ is a lower bound of $p$ and $q$, and
	\item $p \wedge q \geq r$, for any lower bound $r$ of $p$ and $q$.
\end{itemize} Similarly, a \newword{join} of $p$ and $q$ is an element $p \vee q \in \mathcal{P}$ that is an upper bound of $p$ and $q$ and is less than any other such upper bound. Clearly, meets and joins are unique if they exist; indeed, treating $\mathcal{P}$ as a category in the standard way, meet and join coincide with categorical product and coproduct.

\begin{example}\label{ex:meets_and_joins}
In the poset $\cP_{13}$ depicted in \autoref{ex:Hasse}, we have for example that $0103 \vee 0130 = 0133$ and that $1103 \wedge 1013 = 1003$. On the other hand, the meet $1013 \wedge 1130$ does not exist, as no element of the poset lies below both $1013$ and $1130$.
\end{example}

A \newword{lattice} is a poset such that every pair of elements has both a meet and a join. By \autoref{ex:meets_and_joins}, the poset $\cP_{13}$ of \autoref{ex:Hasse} is not a lattice as some meets do not exist; in particular, any two distinct elements of $\cS_\alpha$ do not have a meet. On the other hand, we have the following useful observation about the poset $\widehat{\cP}_\alpha$ which extends $\cP_\alpha$.

\begin{lemma}\label{lem:lattice}
	The extended poset $\widehat{\cP}_\alpha$ is a lattice.
\end{lemma}
\begin{proof}
	Consider any two elements $p,q \in \widehat{\cP}_\alpha$. If $p = \widehat{0}$, then $p \vee q = q$, while if $q = \widehat{0}$, then $p \vee q = p$. Otherwise, we have $p,q \in \cP_\alpha$, so set $p = (p_1, \dots, p_n)$ and $q = (q_1, \dots, q_n)$. By construction, they have a join $s = (s_1, \dots, s_n) \in \widehat{\cP}_\alpha$ given by $s_i = \max \{ p_i, q_i \}$. Thus, $\widehat{\cP}_\alpha$ has all joins.
	
	Since $\widehat{\cP}_\alpha$ is a finite poset with all joins and with a minimum element $\widehat{0}$, it is a lattice by \cite[Proposition~3.3.1]{Stanley:EC1}.
\end{proof}

Consider the alphabet 
\[
\mathbb{B} = \{0,1,2,3, \dots\} \cup \{ \bar{1}, \bar{2}, \bar{3}, \dots \}.
\]
We refer to elements of the first set in $\mathbb{B}$ as \newword{unbarred} and elements of the second as \newword{barred}.
We introduce some local moves that may be performed on strings in the alphabet $\mathbb{B}$:
\begin{itemize}
	\item[(M.1)] \quad $\ldots 0p \ldots \quad \Longrightarrow \quad  \ldots p0 \ldots$, \quad for $p \in \Z_{>0}$;
	\item[(M.2)] \quad $\ldots 0p \ldots \quad \Longrightarrow \quad \ldots p\bar{p} \ldots$, \quad for $p \in \Z_{>0}$.
\end{itemize}
That is, if a string $\sigma$ contains the symbol $0$ followed by an unbarred symbol $p$, we may either swap them by (M.1) or else replace them with the substring $p\bar{p}$ by (M.2).
Let $\widetilde{\cC}_\alpha$ denote the set of strings that can be obtained from elements of $\mathcal{S}_\alpha$ (thought of as strings in $\Z_{\geq 0}$) by iterated application of the local moves (M.1) and (M.2); here, $\cC$ stands for ``combinatorial'' since this set will ultimately be used to construct the aforementioned non-cancellative formula for the quasisymmetric monomial glide $\MK_\alpha$. Note that $\mathcal{S}_\alpha$ is the set of strings obtainable by iterated application of just the local move (M.1), starting from the string $(0^{n-N})^\frown \alpha$, where $0^{n-N}$ denotes a string of $n-N$ zeros and ${}^\frown$ represents concatenation of strings.

Let $\cC_\alpha$ be the set of strings obtained from $\widetilde{\cC}_\alpha$ by replacing each barred symbol with the corresponding unbarred symbol. Note that multiple elements of $\widetilde{\cC}_\alpha$ may yield the same element of $\cC_\alpha$. Equivalently, $\cC_\alpha$ is the set of strings obtainable from elements of $\mathcal{S}_\alpha$ by iterated application of (M.1) and the local move 
\begin{itemize}
	\item[(M.$2'$)] \quad $\ldots 0p \ldots \quad \Longrightarrow \quad \ldots pp \ldots$, \quad for $p \in \Z_{>0}$.
\end{itemize}

\begin{example}\label{ex:CandCtwiddle}
Continuing with $\alpha = (1,3)$ and $n=4$ from \autoref{ex:Hasse}, we have 
\[
\widetilde{\cC}_\alpha = \mathcal{S}_\alpha \cup \{ 01\bar{1}3,1\bar{1}03,10\bar{1}3,1\bar{1}30, 1\bar{1}\bar{1}3,13\bar{3}0,013\bar{3},103\bar{3},130\bar{3},1\bar{1}3\bar{3},13\bar{3}\bar{3}\},
\]
while 
\[
\cC_\alpha = \mathcal{S}_\alpha \cup \{ 0113,1103,1013,1130, 1113,1330,0133,1033,1303,1133,1333\},
\] so that $|\cC_\alpha| = |\widetilde{\cC}_\alpha|$ in this case. Comparing with \autoref{ex:Hasse}, we see that $\cP_\alpha = \cC_\alpha \cup \{1313 \}$.

On the other hand, for $\beta = (1,1)$ and $n=3$, we have 
\[
\mathcal{S}_\beta = \{011,101,110\} \text{ and } \widetilde{\cC}_\beta = \mathcal{S}_\beta \cup \{1 \bar{1} 1, 1 1 \bar{1}\},
\]
while \[
\cC_\beta = \mathcal{S}_\beta \cup \{ 111 \} = \cP_\beta.
\]  In this case, $\cC_\beta$ is strictly smaller than $\widetilde{\cC}_\beta$.
\end{example}

For $\widetilde{s}\in\widetilde{\cC}_\alpha$, let $\barred(\widetilde{s})$ denote the number of barred symbols in $\widetilde{s}$. 

\begin{lemma}\label{lem:subset}
	Viewing $\cP_\alpha$ as a set (forgetting the poset structure), we have 
	\[
\cC_\alpha \subseteq \cP_\alpha.
\]
\end{lemma}
\begin{proof}
For any $\sigma \in \cC_\alpha$, let $r(\sigma)$ be the least nonnegative integer $r$ such that $\sigma$ can be obtained from $\mathcal{S}_\alpha$ by $r$ applications of the local moves (M.1)/(M.$2'$). We induct on $r$.

	Let $\sigma$ be a string of nonnegative integers such that $\sigma \in \cC_\alpha$. If $\sigma \in \mathcal{S}_\alpha$ (i.e., $r(\sigma) = 0$), then $\sigma \in \cP_\alpha$ by definition. Otherwise, choose $\sigma' \in \cC_\alpha$ such that $r(\sigma') = r(\sigma) - 1$, assuming inductively that $\sigma' \in \cP_\alpha$.
	
		Suppose $\sigma$ can be obtained from $\sigma'$ by application of the local move (M.$q$) in positions $i, i+1$, where $q \in \{1, 2' \}$. Then the $i$th entry $\sigma'(i) = 0$ and the $(i+1)$th entry $\sigma'(i+1) = p$ for some positive integer $p$. Since $\sigma' \in \cP_\alpha$, there exist $t\geq1$ and $\tau'_1, \dots, \tau'_t \in \mathcal{S}_\alpha$ such that $\bigvee_{j=1}^t \tau'_j = \sigma'$. In particular, we have that $\tau'_j(i) = 0$ and $\tau'_j(i+1) \leq p$ for all $j$; moreover, $\tau'_j(i+1) = p$ for at least one $j$.

		First suppose $q=1$. For each $j$, let $\tau_j=\tau'_j$ if $\tau'_j(i+1)=0$, and otherwise let $\tau_j$ be obtained from $\tau'_j$ by applying (M.1) in positions $i,i+1$. Each $\tau_j$ belongs to $\mathcal S_\alpha$, and $\bigvee_{j=1}^t\tau_j=\sigma$, so $\sigma\in\cP_\alpha$.

		Now suppose $q=2'$. Let $J=\{j:\tau'_j(i+1)>0\}$. For each $j\in J$, let $\tau_j$ be obtained from $\tau'_j$ by applying (M.1) in positions $i,i+1$, so again $\tau_j\in\mathcal S_\alpha$. The original atoms supply the entry $p$ in position $i+1$, while their shifted copies supply the entry $p$ in position $i$; hence,
		\[
		\sigma=\bigvee_{j=1}^t\tau'_j\vee\bigvee_{j\in J}\tau_j.
		\]
		Thus $\sigma\in\cP_\alpha$ in this case as well.
\end{proof}

Let $\mu_{\cP_\alpha} : \cP_\alpha \to \Z$ be defined by the property that for all $p\in\cP_\alpha$, we have 
\begin{equation}\label{eq:mobius}
\sum_{q\leq p}\mu_{\cP_\alpha}(q)=1.
\end{equation}
 We call $
\mu_{\cP_\alpha}$ the \newword{M\"obius function} on $\cP_\alpha$. Note that, traditionally, the M\"obius function for $\widehat{\cP}_\alpha$ is defined by  $\mu_{\widehat{\cP}_\alpha}(\widehat{0})=1$ and $\sum_{q\leq p}\mu_{\widehat{\cP}_\alpha}(q)=0$ for $p \neq \widehat{0}$. So, to convert back and forth between our convention and the traditional one, we have $\mu_{\cP_\alpha}(p)=-\mu_{\widehat{\cP}_\alpha}(p)$ for all $p\in\cP_\alpha$.
(The reason for this convention is to match a choice from \cite{Allen} which will be applied
in \autoref{sec:K-thy-james-reduced-prod}.) M\"obius functions can be thought of as providing the coefficients in a poset-theoretic generalization of inclusion--exclusion. In \autoref{sec:K-thy-james-reduced-prod}, Knutson's formula will use these coefficients to pass from a union of Schubert varieties to its $K$-class. For background on M\"obius functions on posets, we refer to \cite[Chapter~3]{Stanley:EC1}.

\begin{example}\label{ex:mobius}
	Consider the poset $\cP_{13}$ from \autoref{ex:Hasse}. The values of the M\"obius function $\mu_{\cP_{13}}$ are recorded below.
	 
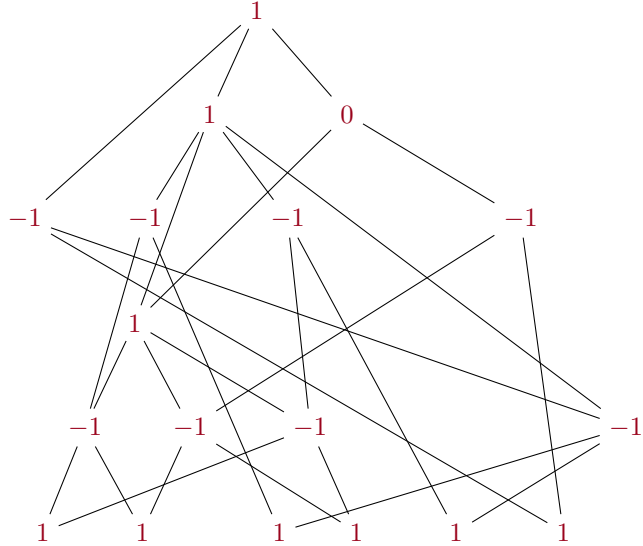
\[
 \begin{tikzpicture}
	\node (0013) at (0,0) {$\textcolor{darkblue}{1}$};
	\node[right= 0.9 of 0013] (0103) {$\textcolor{darkblue}{1}$};
	\node[right= 1.4 of 0103] (0130) {$\textcolor{darkblue}{1}$};
	\node[right= 0.6 of 0130] (1003) {$\textcolor{darkblue}{1}$};
	\node[right= 0.9 of 1003] (1030) {$\textcolor{darkblue}{1}$};
	\node[right= 1.0 of 1030] (1300) {$\textcolor{darkblue}{1}$};
	\node[above right= 0.9 and 0.0 of 0013] (0113) {$\textcolor{darkblue}{-1}$}; %first Ktheory
	\node[right= 0.7  of 0113] (1103) {$\textcolor{darkblue}{-1}$};
	\node[right= 0.9 of 1103] (1013) {$\textcolor{darkblue}{-1}$};
	\node[right = 3.5 of 1013] (1130) {$\textcolor{darkblue}{-1}$}; 
	\node[above right = 0.9 and 0.1 of 0113] (1113) {$\textcolor{darkblue}{1}$}; %rank 3
	\node[above left = 0.9 and 0.9 of 1113] (1330) {$\textcolor{darkblue}{-1}$}; %rank 4
	\node[right = 0.9 of 1330] (0133) {$\textcolor{darkblue}{-1}$}; 
	\node[right = 1.2 of 0133] (1033) {$\textcolor{darkblue}{-1}$}; 
	\node[right = 2.4 of 1033] (1303) {$\textcolor{darkblue}{-1}$}; 
	\node[above right = 0.9 and 1.9 of 1330] (1133) {$\textcolor{darkblue}{1}$}; % rank 5
	\node[right = 1.4 of 1133] (1313) {$\textcolor{darkblue}{0}$}; 
	\node[above right = 0.9 and 0.2 of 1133] (1333) {$\textcolor{darkblue}{1}$}; %top one
\draw (0013) to (0113);
\draw (0013) to (1013);
\draw (0103) to (0113);
\draw (0103) to (1103);
\draw (0130) to (1130);
\draw (0130) to (0133);
\draw (1003) to (1103);
\draw (1003) to (1013);
\draw (1030) to (1130);
\draw (1030) to (1033);
\draw (1300) to (1330);
\draw (1300) to (1303);
\draw (0113) to (0133);
\draw (0113) to (1113);
\draw (1103) to (1113);
\draw (1103) to (1303);
\draw (1013) to (1113);
\draw (1013) to (1033);
\draw (1130) to (1330);
\draw (1130) to (1133);
\draw (1113) to (1133);
\draw (1113) to (1313);
\draw (1330) to (1333);
\draw (0133) to (1133);
\draw (1033) to (1133);
\draw (1303) to (1313);
\draw (1133) to (1333);
\draw (1313) to (1333);
\end{tikzpicture}.
 \]
 These values are computed by labeling each minimal element with $1$ and then working upwards via the recurrence Equation~\eqref{eq:mobius}.
 
 On the other hand, the poset $\cP_{11}$ of \autoref{ex:CandCtwiddle} has diagram
  
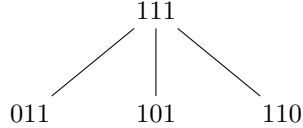
\[
 \begin{tikzpicture}
	\node (011) at (0,0) {$011$};
	\node[right= 0.9 of 011] (101) {$101$};
	\node[right= 0.9 of 101] (110) {$110$};
	\node[above = 0.9  of 101] (111) {$111$}; %first Ktheory
\draw (011) to (111);
\draw (101) to (111);
\draw (110) to (111);
\end{tikzpicture}
 \]
 and M\"obius function $\mu_{\cP_{11}}$ given by
   
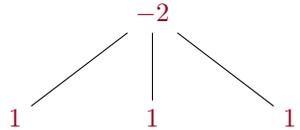
\[
 \begin{tikzpicture}
	\node (011) at (0,0) {$\textcolor{darkblue}{1}$};
	\node[right= 1.4 of 011] (101) {$\textcolor{darkblue}{1}$};
	\node[right= 1.4 of 101] (110) {$\textcolor{darkblue}{1}$};
	\node[above = 0.9  of 101] (111) {$\textcolor{darkblue}{-2}$}; %first Ktheory
\draw (011) to (111);
\draw (101) to (111);
\draw (110) to (111);
\end{tikzpicture}.
 \]
\end{example}

Let
\[
\pi\colon\widetilde{\cC}_\alpha\to\cC_\alpha
\]
be the projection map that forgets barring on all symbols. For $\sigma \in \cC_\alpha$, let
\[
\mu'_{\cC_\alpha}(\sigma)\coloneqq\sum_{\widetilde{\sigma}\in\pi^{-1}(\sigma)}(-1)^{\barred(\widetilde{\sigma})}.
\]
We will show in \autoref{l:muneq0-equals-muprime} that $\mu$ and $\mu'$ agree on the domain $\cC_\alpha$ and in \autoref{prop:mu=0-on-non-Cs} that $\mu$ vanishes off $\cC_\alpha$.

We introduce the following two polynomials. Our primary goal in this section is to prove that they coincide.

\begin{definition}
\label{def:two-polys}
For any tuple $c=(c_1,\dots,c_n)$, let $\y^c\coloneqq y_1^{c_1}\dots y_n^{c_n}$ where the $y_i$ are indeterminates. For any composition $\alpha$, let
\[
\MK_\alpha^\cC(y_1,\dots,y_n)=\sum_{\sigma \in \cC_\alpha}\mu'_{\cC_\alpha}(\sigma)\y^\sigma
\]
and
\[
\MK_\alpha (y_1,\dots,y_n)=\sum_{\sigma \in \cP_\alpha}\mu_{\cP_\alpha}(\sigma) \y^\sigma.
\]
We call $\MK_\alpha (y_1,\dots,y_n)$ the \newword{quasisymmetric monomial glide} for the composition $\alpha$.
\end{definition}

The following theorem is our main result of this section.

\begin{theorem}
\label{thm:two-polys-agree}
For any composition $\alpha$, we have
\[
\MK_\alpha(y_1,\dots,y_n)=\MK_\alpha^\cC(y_1,\dots,y_n).
\]
Furthermore, these polynomials are both quasisymmetric in the variables $y_1,\dots,y_n$.
\end{theorem}

\begin{example}\label{ex:polys}
	Let $\alpha = (1,3)$. Then by comparing \autoref{ex:Hasse}, \autoref{ex:CandCtwiddle}, and \autoref{ex:mobius}, we find that
	\begin{align*}
		\MK_{(1,3)}(y_1,y_2,y_3,y_4) &=\MK_{(1,3)}^\cC(y_1,y_2,y_3,y_4) \\
		&= \y^{0013} + \y^{0103} + \y^{0130} + \y^{1003} + \y^{1030} + \y^{1300} - \y^{0113} - \y^{1103} - \y^{1013} \\
		&-\y^{1130} + \y^{1113} - \y^{1330} - \y^{0133} - \y^{1033} - \y^{1303} + \y^{1133} + \y^{1333}.
	\end{align*}
	Note that these polynomials are equal and quasisymmetric as claimed by \autoref{thm:two-polys-agree}.
\end{example}

By \autoref{lem:subset}, we can prove \autoref{thm:two-polys-agree} by showing that $\mu_{\cP_\alpha}(\sigma) = \mu'_{\cC_\alpha}(\sigma)$ for all $\sigma \in \cC_\alpha$, and that $\mu_{\cP_\alpha}(\sigma) = 0$ for all $\sigma \in \cP_\alpha \smallsetminus \cC_\alpha$. These tasks are accomplished in the next two subsections.

First, we need the following lemma, characterizing $\cC_\alpha$ as a subset of $\cP_\alpha$.

\begin{lemma}\label{lem:C-in-P}
Let $\sigma \in \cP_\alpha$. Then $\sigma \in \cC_\alpha$ if and only if $\sigma$ is of the form $\sigma=\sigma_1\,^\frown\dots^\frown\sigma_k$ where each $\sigma_i$ is a string containing at least $N_i$ instances of $a_i$ with all other symbols being $0$. 
\end{lemma}
\begin{proof}
By \autoref{lem:subset}, we have $\cC_\alpha \subseteq \cP_\alpha$.

Suppose $\sigma\in\cC_\alpha$ and consider some $\widetilde{\sigma}\in\pi^{-1}(\sigma)$. Then by construction, we must have $\widetilde{\sigma}=\widetilde{\sigma}_1\,^\frown\dots^\frown\widetilde{\sigma}_k$ where each $\widetilde{\sigma}_i$ is obtained from the length-$N_i$ string $(a_i,\dots,a_i)$ by inserting additional $0$'s or additional barred symbols $\bar{a}_i$ (the latter not in first position). Thus, $\widetilde{\sigma}_i$ must have $\ell_i\geq N_i$ elements from $\{a_i, \bar{a}_i\}$ (for some $\ell_i$), exactly $\ell_i-N_i$ of them must be barred, and the first nonzero entry of $\widetilde{\sigma}_i$ must be unbarred. In particular, $\sigma$ has the form $\sigma=\sigma_1\,^\frown\dots^\frown\sigma_k$ where each entry of each $\sigma_i$ is either $0$ or $a_i$, and $\sigma_i$ has exactly $\ell_i\geq N_i$ entries equal to $a_i$. Conversely, it is straightforward to see that every such $\sigma$ is in $\cC_\alpha$. 
\end{proof}

\subsection{\texorpdfstring{$\mu_{\cP_\alpha}$}{mu P alpha} vanishes away from \texorpdfstring{$\cC_\alpha$}{C alpha}.}

In this subsection, we establish the following proposition.

\begin{proposition}
\label{prop:mu=0-on-non-Cs}
If $\sigma\in\cP_\alpha\smallsetminus\cC_\alpha$, then $\mu_{\cP_\alpha}(\sigma)=0$.
\end{proposition}

The reader may enjoy confirming directly that \autoref{prop:mu=0-on-non-Cs} holds in the case of \autoref{ex:Hasse} and \autoref{ex:mobius}.
We first handle the following special case, which will be the key to proving the result in general.

\begin{lemma}
\label{l:mu=0-on-non-Cs-distinct}
Suppose that $a_1,\dots,a_k$ are pairwise distinct. If $\sigma \in \cP_\alpha\smallsetminus\cC_\alpha$, then $\mu_{\cP_\alpha}(\sigma)=0$.
\end{lemma}

Our proof of \autoref{l:mu=0-on-non-Cs-distinct} applies Rota's Crosscut Theorem \cite{Rota-Mobius} to express the M\"obius function as a signed sum over subsets $R\subseteq\cS_\alpha$ whose join is $\sigma$. We cancel this signed sum to obtain $0$ by using a sign-reversing involution. In fact, we need two such involutions, one for the case where the sequence has an inversion and one for the case where the sequence is strictly increasing.

The set $\cS_\alpha$ consists of the \newword{atoms} of $\widehat{\cP}_\alpha$, i.e., the minimal elements of $\cP_\alpha$. Every atom $\tau \in \cS_\alpha$ is determined by the positions of its non-zero entries. We record these positions (in increasing order) in a vector $p(\tau)\in\Z_{>0}^N$ and refer to $p(\tau)$ as the \newword{position vector} of $\tau$. For example, if $\tau = (0,1,1,0,3,0,3)$, then its position vector is $p(\tau) = (2,3,5,7)$.

We now construct the two involutions used in the proof of \autoref{l:mu=0-on-non-Cs-distinct}.

\begin{lemma}\label{l:mu-distinct-inversion}
Suppose that $a_1,\dots,a_k$ are pairwise distinct and that all entries of $\sigma\in\cP_\alpha\smallsetminus\cC_\alpha$ are non-zero. Write
\[
\sigma=(\underbrace{a_{i_1},\dots,a_{i_1}}_{\ell_1},\underbrace{a_{i_2},\dots,a_{i_2}}_{\ell_2},\dots,\underbrace{a_{i_K},\dots,a_{i_K}}_{\ell_K}),
\]
where $i_j\neq i_{j+1}$, and let
\[
\cR=\{R\subseteq\cS_\alpha:\bigvee R=\sigma\}.
\]
If there exist $J<L$ with $i_L<i_J$, then there is a fixed-point-free involution $\iota\colon\cR\to\cR$ such that $|\iota(R)|=|R|\pm1$ for all $R\in\cR$.
\end{lemma}

\begin{proof}
The reader may consult \autoref{ex:mu-distinct-inversion} to see the ideas of this argument in a concrete example.

	Choose $J$ and $L$ canonically (depending only on $\sigma$) by requiring that $(J,L)$ is the smallest possible in lexicographical order. %We will construct $\iota$ depending on this choice of $J$ and $L$.

Let $\cS_- \subseteq \cS_\alpha$ be the set of atoms $\tau$ %whose $i_J$-th non-zero entry occurs in position $(\ell_1+\dots+\ell_{J-1},\ell_1+\dots+\ell_J]$. In other words, these are the atoms
lying below $\sigma$ with an $a_{i_J}$ occurring in the same position of $\tau$ as an $a_{i_J}$-entry of $\sigma$;
that is,
\[
\cS_- = \{ \tau \in \cS_\alpha : \tau \leq \sigma \text{ and there exists $h \in [1 + \sum_{g=1}^{J-1} \ell_g, \sum_{g=1}^{J} \ell_g]$ with } \tau(h) = \sigma(h) = a_{i_J} \}.
\]
 Similarly, let $\cS_+ \subseteq \cS_\alpha$ be the set of atoms $\tau$ lying under $\sigma$ where an $a_{i_L}$ occurs in the same position as an $a_{i_L}$-entry of $\sigma$; that is,
 \[
\cS_+ = \{ \tau \in \cS_\alpha : \tau \leq \sigma \text{ and there exists $h \in [1 + \sum_{g=1}^{L-1} \ell_g, \sum_{g=1}^{L} \ell_g]$ with } \tau(h) = \sigma(h) = a_{i_L} \}.
\]
Note that $\cS_\pm$ are also canonically associated to $\sigma$.

Let $R\in\cR$. Since the $a_i$ are pairwise distinct, we have $R\cap\cS_-\neq\varnothing$ and $R\cap\cS_+\neq\varnothing$. Let $u_-\in R\cap\cS_-$ be the atom whose position vector $p(u_-)$ is least in lexicographical order. Similarly, let $u_+\in R\cap\cS_+$ be the atom whose position vector $p(u_+)$ is greatest in lexicographical order. 

Notice that $u_\pm$ are canonically associated to $R$ and that $u_-\neq u_+$ since the locations of their $a_{i_L}$-entries differ. Indeed, $u_+$ has an $a_{i_L}$ appearing in the interval $[1 + \sum_{g=1}^{L-1} \ell_g, \sum_{g=1}^{L} \ell_g]$; in contrast, since $u_-$ is an atom, all of its $a_{i_L}$ appear left of every $a_{i_J}$ and in particular left of the one in the interval $[1 + \sum_{g=1}^{J-1} \ell_g, \sum_{g=1}^{J} \ell_g]$.

Let $u\coloneqq u_R \in \cS_\alpha$ be the atom defined by ``splicing together'' $u_-$ and $u_+$ as follows. %We define the position vector $p(u)$ of $u$ to agree with $p(u_-)$ for the first $i_L$ non-zero positions, and to agree with $p(u_+)$ for the remaining positions. In other words,
For $i\leq i_L$, $u$ has the values $a_i$ in the same positions as $u_-$; for $i>i_L$, $u$ has the values $a_i$ in the same positions as $u_+$. Note that $u$ is well-defined, since each position where $u_-$ has a value $a_i$ with $i\leq i_L$ is strictly left of each position where $u_+$ has a value $a_j$ with $j>i_L$; moreover, $u$ is clearly an atom. Also observe that $u\neq u_-$ since their $a_{i_J}$-positions differ, and $u\neq u_+$ since their $a_{i_L}$-positions differ.

Since $u$ is canonically associated to $R$, we may define
\begin{equation}\label{eq:iota}
\iota(R)=
\begin{cases}
R\cup\{u\}, & u\notin R;\\
R\smallsetminus\{u\}, & u\in R,
\end{cases}
\end{equation}
a sort of \emph{toggle} in the sense of \cite[Definition~2.1]{Striker}.
Furthermore, since 
\[
u\vee u_-\vee u_+=u_-\vee u_+,
\] we see that 
\[
\bigvee\iota(R)=\bigvee R=\sigma.
\]
 In other words, $\iota(R)\in\cR$. Lastly, since $\iota(R)$ differs from $R$ exactly by the element $u$, and $u$ agrees with $u_-$ (respectively $u_+$) for $i\leq i_L$ (respectively $i>i_L$), we see $u_R=u_{\iota(R)}$; hence, $\iota(\iota(R))=R$ and $\iota$ is an involution. By construction, $|\iota(R)|=|R|\pm1$.
\end{proof}

\begin{example}\label{ex:mu-distinct-inversion}
Suppose that $\alpha = (1,3)$, $n=8$, and $\sigma = 11331133$. Then $a_1 = 1$ and $a_2 = 3$.  Moreover, \[i_1 = 1, i_2 = 2, i_3 = 1, \text{ and } i_4 = 2.\] Hence, we have $i_3 < i_2$ and can only choose $J = 2$ and $L=3$, so $a_{i_J} = 3$ and $a_{i_L} = 1$.
An example of $R \in \cR$ is
\[
R = \{ 10300000, 01030000,00001030,00000103, 00000013 \}.
\]
In this case, we have
\[
R \cap \cS_- = \{ \tau \in R : \text{there exists } h \in [3,4] \text{ with } \tau(h) = \sigma(h) = 3 \} = \{ 10300000, 01030000 \}
\]
and
\[
R \cap \cS_+ = \{ \tau \in R : \text{there exists } h \in [5,6] \text{ with } \tau(h) = \sigma(h) = 1 \}  = \{00001030, 00000103 \}.
\]
Therefore, $u_- = 10300000$ and $u_+ = 00000103$. Splicing these atoms together as in the proof of \autoref{l:mu-distinct-inversion} gives $u_R = 10000003$.
\end{example}

\begin{lemma}\label{l:mu-distinct-gap}
Suppose that $a_1,\dots,a_k$ are pairwise distinct and that all entries of $\sigma\in\cP_\alpha\smallsetminus\cC_\alpha$ are non-zero. Write
\[
\sigma=(\underbrace{a_{i_1},\dots,a_{i_1}}_{\ell_1},\underbrace{a_{i_2},\dots,a_{i_2}}_{\ell_2},\dots,\underbrace{a_{i_K},\dots,a_{i_K}}_{\ell_K}),
\]
where $i_1<i_2<\dots<i_K$, and let
\[
\cR=\{R\subseteq\cS_\alpha:\bigvee R=\sigma\}.
\]
Then there is a fixed-point-free involution $\iota\colon\cR\to\cR$ such that $|\iota(R)|=|R|\pm1$ for all $R\in\cR$.
\end{lemma}

\begin{proof}
Again, the reader may consult \autoref{ex:mu-distinct-gap} to see the ideas of this argument in a concrete example.

As in the proof of \autoref{l:mu=0-on-non-Cs-distinct}, the first and final entries of $\sigma$ show that $i_1=1$ and $i_K=k$. Since $\sigma \in \cP_\alpha\smallsetminus\cC_\alpha$, \autoref{lem:C-in-P} shows that there is a least $J \in [k-1]$ such that
\[
i_1=1,\quad\dots\quad, i_J=J,\quad\textrm{and }\quad m\coloneqq i_{J+1}> J+1.
\]
We follow a similar strategy as in the proof of \autoref{l:mu-distinct-inversion}. Let $\cS_-$ be the set of atoms $\tau\leq\sigma$ such that $\tau(\ell_1+\dots+\ell_J+1)=a_m$. In other words, these are the atoms below $\sigma$ where the value $a_m$ occurs in the same position as the first $a_m$ entry of $\sigma$. Similarly, let $\cS_+$ be the atoms lying below $\sigma$ where $a_J$ occurs in the same position as the last $a_J$ of $\sigma$.

Let $R\in\cR$. Again, we have $R\cap\cS_-\neq\varnothing$ and $R\cap\cS_+\neq\varnothing$. As in the proof of \autoref{l:mu-distinct-inversion}, let $u_-\in R\cap\cS_-$ be the atom whose position vector $p(u_-)$ is least in lexicographical order and let $u_+\in R\cap\cS_+$ be the atom whose position vector $p(u_+)$ is greatest in lexicographical order. Since $J+1<m$, we see the positions of the $a_{J+1}$-entries differ in $u_\pm$, and so $u_-\neq u_+$.

Let $u\coloneqq u_R$ be the atom defined as follows. For $i\leq J$, we let $u$ have the values $a_i$ in the same positions as does $u_-$; for $i>J$, we let $u$ have the values $a_i$ in the same positions as does $u_+$. Then $u\neq u_-$ since their $a_{J+1}$-positions differ, and $u\neq u_+$ since their $a_J$-positions differ. We can therefore again define $\iota(R)$ by Equation~\eqref{eq:iota}. Once again, $u\vee u_-\vee u_+=u_-\vee u_+$, and so $\iota(R)\in\cR$. We also have $u_R=u_{\iota(R)}$, so $\iota(\iota(R))=R$. Clearly, $|\iota(R)|=|R|\pm1$.
\end{proof}

\begin{example}\label{ex:mu-distinct-gap}
Suppose that $\alpha = (5,4,7,1)$, $n=5$, and $\sigma = 55771$. Then, $J=1$ and $m = 3$. We have
\[
\cS_- = \{ 54701,54710\}\quad\textrm{and}\quad
\cS_+ = \{05471\}.
\]
A choice of $R \in \cR$ is
\[
R = \{ 05471, 54701 \}.
\]
With this choice, we have
\[
R \cap \cS_- = \{54701 \} \text{ and } R \cap \cS_+ = \{05471\}.
\]
Then, $u_- = 54701$ and $u_+ = 05471$. Splicing these atoms together as in the proof of \autoref{l:mu-distinct-gap} gives $u = 50471$.
\end{example}

\begin{proof}[{Proof of \autoref{l:mu=0-on-non-Cs-distinct}}]
Fix $\sigma \in \cP_\alpha\smallsetminus\cC_\alpha$.
Since $\widehat{\cP}_\alpha$ is a lattice by \autoref{lem:lattice}, Rota's Crosscut Theorem \cite{Rota-Mobius} tells us that if
\[
\cR =\{R \subseteq \cS_\alpha : \bigvee R = \sigma \},
\]
then
\[
-\mu_{\cP_\alpha}(\sigma)=\mu_{\widehat{\cP}_\alpha}(\sigma)=\sum_{R\in\cR}(-1)^{|R|}.
\]
To prove $\mu_{\cP_\alpha}(\sigma)=0$, it suffices to exhibit a fixed-point-free involution
\[
\iota\colon\cR\to\cR
\]
with the property that $|\iota(R)|=|R|\pm1$ for all $R \in \cR$. Indeed, then we have
\[
\sum_{R\in\cR}(-1)^{|R|}=\sum_{R\in\cR}(-1)^{|\iota(R)|}=-\sum_{R\in\cR}(-1)^{|R|},
\]
so that this sum, and hence $\mu_{\cP_\alpha}(\sigma)$, is zero.

Suppose the zero entries of $\sigma$ occur in positions $z_1<\dots<z_M$. Note that if $\sigma=\bigvee R$, then every $\tau \in R$ has $\tau \leq \sigma$; in particular, $\tau$ must also have zeros in positions $z_1<\dots<z_M$ (and potentially additional positions). Deleting these positions, and reinserting them afterward, identifies $\cR$ with the corresponding family for the resulting zero-free string and preserves joins and cardinalities; by \autoref{lem:C-in-P}, it also preserves whether $\sigma$ belongs to $\cC_\alpha$. Therefore, restricting attention to the positions $[n]\smallsetminus\{z_1,\dots,z_M\}$, we may suppose $\sigma$ has no $0$ entries at all. We may therefore write
\[
\sigma=(\underbrace{a_{i_1},\dots,a_{i_1}}_{\ell_1},\underbrace{a_{i_2},\dots,a_{i_2}}_{\ell_2},\dots,\underbrace{a_{i_K},\dots,a_{i_K}}_{\ell_K})
\]
where $i_j\neq i_{j+1}$. Since the first entry of every atom $\tau \in \cS_\alpha$ is either $0$ or $a_1$, and since $\sigma$ is assumed to have only non-zero entries, we see the first entry of $\sigma$ is $a_1$. Similarly, the final entry of $\sigma$ must be $a_k$. That is, we have
\[
i_1=1\quad\textrm{and}\quad i_K=k.
\]
	The sequence $i_1,\dots,i_K$ either has an inversion, meaning there exist $J<L$ with $i_L<i_J$, or else is strictly increasing. In the first case, \autoref{l:mu-distinct-inversion} gives the required involution, while \autoref{l:mu-distinct-gap} gives it in the second case.
\end{proof}

We will prove \autoref{prop:mu=0-on-non-Cs} by reducing to the special case covered by \autoref{l:mu=0-on-non-Cs-distinct}. First, let us rewrite Equation~\eqref{eqn:alpha-underbrace-expression} in different notation. Let
\[
\alpha=(\alpha_1,\dots,\alpha_N);
\]
here, we do not assume the $\alpha_i$ are pairwise distinct, and we even allow the possibility that $\alpha_i=\alpha_{i+1}$. Let $\omega\in S_N$ be the unique permutation with the properties that
\[
\alpha_{\omega(1)}\leq \alpha_{\omega(2)}\leq\dots\leq \alpha_{\omega(N)}
\]
and if $\alpha_{\omega(i)}=\alpha_{\omega(i+1)}$, then $\omega(i)<\omega(i+1)$; that is to say, $\omega$ is the shortest permutation that sorts $\alpha$. Let
\[
\beta=(\beta_1,\dots,\beta_N),\quad\textrm{where}\quad \beta_i=\omega^{-1}(i).
\]

For example, if $\alpha = (2,1,3,1)$, then we have $\omega = 2413$, where we write a permutation $w \in S_N$ in \newword{one-line notation} $w = w(1)w(2)\ldots w(N)$, for then $(\alpha_{\omega(1)}, \alpha_{\omega(2)}, \alpha_{\omega(3)}, \alpha_{\omega(4)}) = (1,1,2,3)$. Therefore, $\omega^{-1} = 3142$, and so $\beta = (\beta_1, \beta_2, \beta_3, \beta_4) = (3,1,4,2)$.

Then we have two posets $\cP_\alpha$ and $\cP_\beta$ of strings of length $n$. Recall that $\cS_\alpha$ (respectively $\cS_\beta$) is the set of atoms of $\widehat{\cP}_\alpha$ (respectively $\widehat{\cP}_\beta$). Since an atom is determined by the positions of its nonzero entries, we have mutually-inverse bijections
\[
\xymatrix{
 \cS_\alpha \ar@<.5ex>[r]^-{\st}  & \cS_\beta\ar@<.5ex>[l]^-{\des}
}
\]
where $\st$ replaces the $i$th nonzero entry with $\beta_i$, and $\des$ replaces the $i$th nonzero entry with $\alpha_i$. We refer to $\st$ as \newword{standardization} and $\des$ as \newword{semistandardization}, since they are analogous to the classical (semi)standardization bijections on Young tableaux (see, e.g., \cite{Stanley:EC2,Kirillov,Pechenik.Yong:genomic} for discussion of these maps in that setting). These maps induce bijections on the power sets of $\cS_\alpha$ and $\cS_\beta$, which we continue to call the standardization and semistandardization maps. We use superscripts to denote these functions, e.g., for $\tau \in\cS_\alpha$ and $T\subseteq \cS_\alpha$, we write $\tau^{\st} \in \cS_\beta$ and $T^{\st} \coloneqq \{ \theta^{\st} : \theta \in T \} \subseteq \cS_\beta$ for their standardizations.

\begin{lemma}
\label{l:bigveedesinC}
If $T\subseteq \cS_\beta$ and $\bigvee T\in\cC_\beta$, then $\bigvee(T^{\des})\in\cC_\alpha$.
\end{lemma}
\begin{proof}
Let $\tau \coloneqq  \bigvee T$. Suppose the zero entries of $\tau$ occur in positions $z_1<\dots<z_M$. Then every atom in $T$ has zeros in positions $z_1<\dots<z_M$ (and potentially additional zeros), which implies that every atom in $T^{\des}$ has zeros in positions $z_1<\dots<z_M$ as well. Therefore, restricting attention to the positions $[n]\smallsetminus\{z_1,\dots,z_M\}$, we may suppose $\tau$ has no zero entries at all. 

By \autoref{lem:C-in-P}, we then know
\[
\tau=(\underbrace{\beta_1,\dots,\beta_1}_{\ell_1},\underbrace{\beta_2,\dots,\beta_2}_{\ell_2},\dots,\underbrace{\beta_N,\dots,\beta_N}_{\ell_N})
\]
with each $\ell_i>0$. This implies that in the positions of the interval $[1+\sum_{i=1}^{j-1}\ell_i,\sum_{i=1}^{j}\ell_i]$, the only allowable nonzero entries for atoms in $T$ are $1,2,\dots,\beta_j$; moreover at every position in this range, there exists some atom $\theta \in T$ whose value at this position is $\beta_j$. By semistandardization, we learn that in positions $[1+\sum_{i=1}^{j-1}\ell_i,\sum_{i=1}^{j}\ell_i]$, the only allowable nonzero entries for atoms in $T^{\des}$ are $\alpha_{\omega(1)},\alpha_{\omega(2)},\dots,\alpha_{\omega(\omega^{-1}(j))}=\alpha_j$; moreover at every position in this range, there exists some atom $\theta^{\des} \in T^{\des}$ whose value at this position is $\alpha_j$. Since \[
\alpha_{\omega(1)}\leq \alpha_{\omega(2)}\leq\dots\leq \alpha_j,
\]
 it follows that 
\[
\bigvee(T^{\des})=(\underbrace{\alpha_1,\dots,\alpha_1}_{\ell_1},\underbrace{\alpha_2,\dots,\alpha_2}_{\ell_2},\dots,\underbrace{\alpha_N,\dots,\alpha_N}_{\ell_N}),
\]
which is in $\cC_\alpha$ by \autoref{lem:C-in-P}.
\end{proof}

We now turn finally to proving \autoref{prop:mu=0-on-non-Cs}.

\begin{proof}[{Proof of \autoref{prop:mu=0-on-non-Cs}}]
Fix $\sigma \in\cP_\alpha\smallsetminus\cC_\alpha$. Let $\cR$ be the collection of all subsets $R \subset\cS_\alpha$ such that $\bigvee R=\sigma$. As in the proof of \autoref{l:mu=0-on-non-Cs-distinct}, it suffices to construct a fixed-point-free involution
\[
\jmath\colon\cR\to\cR
\]
such that $|\jmath(R)|=|R|\pm1$ for all $R \in \cR$. Note that if $\bigvee(R^{\st})\in\cC_\beta$, then \autoref{l:bigveedesinC} shows $\bigvee((R^{\st})^{\des})=\bigvee R=\sigma \in \cC_\alpha$, contradicting the hypothesis on $\sigma$. Therefore
\[
\tau_R\coloneqq \bigvee(R^{\st})\in\cP_\beta\smallsetminus\cC_\beta.
\]
For each $R \in \cR$, let $\cT_R=\{T\subseteq \cS_\beta\mid \bigvee T=\tau_R\}$ and let
\[
\iota_R\colon\cT_R\to\cT_R
\]
denote the involution obtained, after suppressing the zero positions of $\tau_R$, from \autoref{l:mu-distinct-inversion} or \autoref{l:mu-distinct-gap}, as appropriate. We define
\[
\jmath(R)\coloneqq (\iota_R(R^{\st}))^{\des}.
\]

We first show $\jmath(R)\in\cR$, i.e., $\bigvee\jmath(R)=\bigvee R$. For this, we recall from the proofs of \autoref{l:mu-distinct-inversion} and \autoref{l:mu-distinct-gap} that, in either case, associated to $R^{\st}$, we have two distinct atoms $u_\pm\in R^{\st}$. We then form a third atom $u$ by splicing together partial data from $u_-$ and $u_+$; specifically, there is an index $I$ such that, for each $j\leq I$, the symbol $\beta_j$ occupies in $u$ the same position it occupies in $u_-$, while for each $j>I$, the symbol $\beta_j$ occupies in $u$ the same position it occupies in $u_+$. Then $\iota_R(R^{\st})=R^{\st}\cup\{u\}$ if $u\notin R^{\st}$, and $\iota_R(R^{\st})=R^{\st}\smallsetminus\{u\}$ if $u\in R^{\st}$. Therefore, we see that $\jmath(R)=R\cup\{u^{\des}\}$ if $u^{\des}\notin R$, and $\jmath(R)=R\smallsetminus\{u^{\des}\}$ if $u^{\des}\in R$. Semistandardization changes the symbols but not their positions, so every nonzero entry of $u^{\des}$ agrees with the entry in the same position of either $u_-^{\des}$ or $u_+^{\des}$. Hence, $u^{\des}\leq u_-^{\des}\vee u_+^{\des}$ and therefore $u^{\des}\vee u_-^{\des}\vee u_+^{\des}=u_-^{\des}\vee u_+^{\des}$, which implies $\bigvee\jmath(R)=\bigvee R$.

Next, since $\st$ and $\des$ are mutually-inverse bijections of atoms, we have
\[
|\jmath(R)| = |(\iota_R(R^{\st}))^{\des}|=|\iota_R(R^{\st})|=|R^{\st}|\pm1=|R|\pm1.
\]

Finally, we must show $\jmath$ is an involution. Directly from the definitions, we see
\[
\jmath(\jmath(R))=(\iota_{R'}\iota_R(R^{\st}))^{\des},
\]
where $R'=(\iota_R(R^{\st}))^{\des}$. Since
\[
\tau_{R'}\coloneqq \bigvee(R')^{\st}=\bigvee \iota_R(R^{\st})=\bigvee R^{\st}=:\tau_R,
\]
we see that $\cT_R=\cT_{R'}$. Furthermore, since $\iota_R$ (respectively $\iota_{R'}$) depends only on $\tau_R$ (respectively $\tau_{R'}$), we see
\[
\jmath(\jmath(R))=(\iota_R\iota_R(R^{\st}))^{\des}=(R^{\st})^{\des}=R,
\]
as desired.
\end{proof}

\subsection{Proof of \autoref{thm:two-polys-agree}}

Having shown in \autoref{prop:mu=0-on-non-Cs} that $\mu_{\cP_\alpha}$ vanishes away from $\cC_\alpha$, it remains to compute $\mu_{\cP_\alpha}$ on the subset $\cC_\alpha$.

\begin{lemma}
\label{l:muneq0-equals-muprime}
If $\sigma\in\cC_\alpha$, then $\mu_{\cP_\alpha}(\sigma)=\mu'_{\cC_\alpha}(\sigma)$. As in \autoref{lem:C-in-P}, let $\sigma=\sigma_1\,^\frown\dots^\frown\sigma_k$ where each $\sigma_i$ is a string containing at least $N_i$ instances of $a_i$ with all other symbols being $0$. If $\sigma_i$ has exactly $\ell_i$ entries equal to $a_i$, then 
\[
\mu_{\cP_\alpha}(\sigma)=(-1)^{\ell_1+\dots+\ell_k-N}\prod_{i=1}^k \binom{\ell_i-1}{N_i-1}=\mu'_{\cC_\alpha}(\sigma).
\]
\end{lemma}
\begin{proof}
As in \autoref{lem:C-in-P}, if $\sigma\in\cC_\alpha$ and $\widetilde{\sigma}\in\pi^{-1}(\sigma)$, then $\widetilde{\sigma}$ has the form $\widetilde{\sigma}=\widetilde{\sigma}_1\,^\frown\dots^\frown\widetilde{\sigma}_k$, where each $\widetilde{\sigma}_i$ has $\ell_i\geq N_i$ elements from $\{a_i, \bar{a}_i\}$ and exactly $\ell_i-N_i$ must be barred. Moreover, the first nonzero entry of $\widetilde{\sigma}_i$ must be unbarred, and every subset of size $\ell_i-N_i$ of the remaining $\ell_i-1$ nonzero entries can be barred in $\widetilde{\sigma}_i$ for some $\widetilde{\sigma} \in \pi^{-1}(\sigma)$. Thus, there are $\binom{\ell_i-1}{\ell_i-N_i}=\binom{\ell_i-1}{N_i-1}$ possibilities for $\widetilde{\sigma}_i$ and each of these has $\barred(\widetilde{\sigma}_i)=\ell_i-N_i$. Thus,
\[
\mu'_{\cC_\alpha}(\sigma)=(-1)^{\ell_1+\dots+\ell_k-N}\prod_{i=1}^k \binom{\ell_i-1}{N_i-1}.
\]

It remains to prove that $\mu_{\cP_\alpha}(\sigma)$ is also computed by the same product of binomial coefficients. Extend the domain of the function $\mu'_{\cC_\alpha}$ from $\cC_\alpha$ to $\cP_\alpha$ by declaring that $\mu'_{\cC_\alpha}$ is zero on $\cP_\alpha\smallsetminus\cC_\alpha$. By \autoref{prop:mu=0-on-non-Cs}, we have $\mu_{\cP_\alpha}(\theta)=\mu'_{\cC_\alpha}(\theta)$ for $\theta\in\cP_\alpha\smallsetminus\cC_\alpha$. We prove the equality for $\theta\in\cC_\alpha$ by induction through $\cP_\alpha$. Fix $\theta\in\cC_\alpha$ and assume the equality holds for every $\tau<\theta$. By \eqref{eq:mobius}, it suffices to show that
\begin{equation}\label{eqn:inductivemuprime}
\sum_{\tau\leq \theta}\mu'_{\cC_\alpha}(\tau)=1.
\end{equation}
The inductive hypothesis then identifies all the terms below $\theta$ in the recurrences for $\mu'_{\cC_\alpha}$ and $\mu_{\cP_\alpha}$.

First, suppose the zero entries of $\theta$ occur in positions $z_1<\dots<z_M$. Then note that if $\tau\leq \theta$, then $\tau$ must also have zeros in positions $z_1<\dots<z_M$. Restricting attention to the positions $[n]\smallsetminus\{z_1,\dots,z_M\}$, we may therefore suppose $\theta$ has no zero entries at all. Next, in order to prove that Equation~\eqref{eqn:inductivemuprime} holds, it suffices to prove that
 \[
\sum_{\substack{\tau\leq \theta \\ \tau\in\cC_\alpha}}\mu'_{\cC_\alpha}(\tau)=1,
\]
 since for $\tau \notin \cC_\alpha$, we have $\mu'_{\cC_\alpha}(\tau)=0$ by definition.

To handle the base case, suppose $\theta \in \cS_\alpha$ is minimal. Then $\theta$ is obtained from $\alpha$ by inserting $n-N$ zero entries. Therefore, $\theta$ has precisely $\ell_i=N_i$ entries equal to $a_i$, so $\mu'_{\cC_\alpha}(\theta)=(-1)^{N-N}\prod_{i=1}^k \binom{N_i-1}{N_i-1}=1=\mu_{\cP_\alpha}(\theta)$.

For the induction step, suppose that $\theta$ has $\ell_i$ entries equal to $a_i$. Then any $\tau\leq \theta$ with $\tau \in \cC_\alpha$ is obtained from $\theta$ by picking some $j_i$ with $N_i\leq j_i\leq\ell_i$, choosing $j_i$ of the $a_i$-entries, and replacing the $\ell_i-j_i$ remaining $a_i$-entries with zeros. Thus, proving Equation~\eqref{eqn:inductivemuprime} amounts to showing
\begin{equation}\label{eqn:binomial-identity-we-need-to-show}
\sum_{N_i\leq j_i\leq\ell_i}(-1)^{j_1+\dots+j_k-N}\prod_{i=1}^k \binom{j_i-1}{N_i-1}\binom{\ell_i}{j_i}=1.
\end{equation}

Since 
\[
\sum_{N_i\leq j_i\leq\ell_i}(-1)^{j_1+\dots+j_k-N}\prod_{i=1}^k \binom{j_i-1}{N_i-1}\binom{\ell_i}{j_i}
\quad=\quad 
\prod_{i=1}^k\sum_{N_i\leq j_i\leq\ell_i}(-1)^{j_i-N_i}\binom{j_i-1}{N_i-1}\binom{\ell_i}{j_i},
\]
it suffices to prove Equation~\eqref{eqn:binomial-identity-we-need-to-show} in the case where $k=1$. That is to say, we will prove the identity
\begin{equation}\label{eq:binomial_identity}
\sum_{j=N}^\ell(-1)^{j-N} \binom{j-1}{N-1}\binom{\ell}{j}=1
\end{equation}
by induction on $\ell$. When $\ell=N$, we see that Equation~\eqref{eq:binomial_identity} holds. For $\ell>N$, we use the Pascal's triangle identity $\binom{\ell}{j}=\binom{\ell-1}{j-1}+\binom{\ell-1}{j}$ to rewrite the left side of Equation~\eqref{eq:binomial_identity} as 
\begin{equation}\label{eq:pascal}
\sum_{j=N}^\ell(-1)^{j-N} \binom{j-1}{N-1}\binom{\ell-1}{j-1} \quad +\quad \sum_{j=N}^{\ell-1}(-1)^{j-N} \binom{j-1}{N-1}\binom{\ell-1}{j}.
\end{equation}
By the inductive hypothesis, the second sum of \eqref{eq:pascal} equals $1$, so it remains to show that the first sum of \eqref{eq:pascal} vanishes. For this, we rewrite it as 
\[
\sum_{j=N}^\ell(-1)^{j-N} \binom{j-1}{N-1}\binom{\ell-1}{j-1}\ =\ \binom{\ell-1}{N-1}\sum_{j=N}^\ell(-1)^{j-N} \binom{\ell-N}{j-N}\ =\  \binom{\ell-1}{N-1}(1-1)^{\ell-N}\ =\ 0.\qedhere
\]
\end{proof}

We now prove the main theorem of this section.

\begin{proof}[{Proof of \autoref{thm:two-polys-agree}}]
\autoref{prop:mu=0-on-non-Cs} shows that if $\sigma \in\cP_\alpha\smallsetminus\cC_\alpha$, then $\mu_{\cP_\alpha}(\sigma)=0$. On the other hand, by \autoref{l:muneq0-equals-muprime}, if $\sigma \in\cC_\alpha$, then $\mu'_{\cC_\alpha}(\sigma)=\mu_{\cP_\alpha}(\sigma)$. So, by definition, we then have \[
\MK_\alpha^\cC(y_1,\dots,y_n) = \MK_\alpha(y_1,\dots,y_n),\]
 as desired.

Lastly, suppose $\sigma$ and $\sigma'$ are weak compositions of length $n$ such that $\sigma^+ = (\sigma')^+$, so that they differ only in the location of zeros. The construction of $\cC_\alpha$, equivalently the proof of the characterization in \autoref{lem:C-in-P}, shows that $\sigma \in\cC_\alpha$ if and only if $\sigma'\in\cC_\alpha$. If neither belongs to $\cC_\alpha$, then the coefficients in $\MK_\alpha^\cC(y_1,\dots,y_n)$ of $y^\sigma$ and $y^{\sigma'}$ are both zero. 
Otherwise, the integers $\ell_i$ appearing in \autoref{l:muneq0-equals-muprime} are the same for $\sigma$ and $\sigma'$, so the explicit formula there shows their coefficients are equal. It follows then that $\MK_\alpha^\cC(y_1,\dots,y_n)$ is a quasisymmetric polynomial. Since $\MK_\alpha(y_1,\dots,y_n) = \MK_\alpha^\cC(y_1,\dots,y_n)$, the polynomial $\MK_\alpha(y_1,\dots,y_n)$ is also quasisymmetric.
\end{proof}

\begin{remark}\label{rem:LP}
	T.~Lam and P.~Pylyavskyy introduced inhomogeneous deformations of monomial quasisymmetric functions in \cite[Remark~5.14]{Lam.Pylyavskyy}. They refer to these power series as \emph{multimonomial quasisymmetric functions} $\tilde{M}_\alpha$ and consider them to be ``$K$-theoretic'' through analogy with certain other inhomogeneous power series related to the $K$-theory of Grassmannians. To our knowledge, multimonomial quasisymmetric functions have never been studied outside of that remark; in particular, no explicit connection between multimonomial quasisymmetric functions and $K$-theoretic geometry has been obtained.
	
	The multimonomial quasisymmetric functions $\tilde{M}_\alpha$ are similar to, but different from our monomial quasisymmetric glides $\MK_\alpha$. (There appears to be a typo in the definition of \cite[Remark~5.14]{Lam.Pylyavskyy}; assuming that the lowest-degree piece of $\tilde{M}_\alpha$ should be the monomial quasisymmetric function, as the authors suggest, the ``$\subset$'' in their definition should be ``$\supset$''; otherwise, the power series $\MK_\alpha$ and $\tilde{M}_\alpha$ are even more different.) For example, suppose that $\alpha = (1,3)$. Then, truncating to four variables, we have 
		\begin{align*}
		\MK_{(1,3)}(y_1,y_2,y_3,y_4) 		= \y^{0013} &+ \y^{0103} + \y^{0130} + \y^{1003} + \y^{1030} + \y^{1300} - \y^{0113} - \y^{1103} - \y^{1013} \\
		&-\y^{1130} + \y^{1113} - \y^{1330} - \y^{0133} - \y^{1033} - \y^{1303} + \y^{1133} + \y^{1333}
	\end{align*}
	(as in \autoref{ex:polys}), while 
		\begin{align*}
		\tilde{M}_{(1,3)}(y_1,y_2,y_3,y_4) 		= \y^{0013} &+ \y^{0103} + \y^{0130} + \y^{1003} + \y^{1030} + \y^{1300} - \y^{1310} - \y^{1301} - \y^{1220} \\
		&-\y^{1202} -2 \y^{1130} - 2\y^{1103} - \y^{1031} - \y^{1022} - 2\y^{1013} - \y^{0131} - \y^{0122} \\
		&-2\y^{0113} + \y^{1311} + \y^{1221} + \y^{1212} + 2\y^{1131} + 2\y^{1122} + 3\y^{1113}.
	\end{align*}
\end{remark}

\begin{remark}
	Following \cite{Assaf.Searles}, there has been a flurry of recent work on lifting important bases of $\QSym_n$ to the full polynomial ring $\Z[y_1, \dots, y_n]$ (see, e.g., \cite{Assaf.Searles:kohnertTab,Pechenik.Searles, Pechenik.Searles:survey, Mason.Searles}).
	
	In this vein, we note that the quasisymmetric monomial glides $\MK_\alpha(y_1, \dots, y_n)$ have a natural such lift. Specifically, for a weak composition $a=(a_1, \dots, a_n)$, one may define the \newword{monomial glide} $\MK_a(y_1, \dots, y_n)$ to be the generating function for strings obtained from $a$ via the local moves (M.1) and (M.2). That is, letting $\cC_a$ denote the set of such obtainable strings, we define
	\[
	\MK_a(y_1, \dots, y_n) = \sum_{\sigma \in \cC_a} \mu'_{\cC_a}(\sigma) \y^\sigma.
	\]
	Note that this definition recovers the quasisymmetric monomial glides, as we have \[
	\MK_\alpha(y_1, \dots, y_n) = \MK_{({0^{n-\ell(\alpha)}})^\frown \alpha }(y_1, \dots, y_n);
	\]
	moreover, the monomial glides that are quasisymmetric are exactly those of this form. We do not pursue these generalizations further in this paper.	
\end{remark}

\section{The \texorpdfstring{$K$}{K}-theory of the James reduced product}
\label{sec:K-thy-james-reduced-prod}

This section has two main steps. First, we compute the completed rational $K$-theory of $J(\C\bP^\infty)$ by finite approximations and identify it with $\mQSymQ$. We then construct a $K$-class associated to each cell and use Knutson's M\"obius-inversion formula to identify it with the corresponding quasisymmetric monomial glide.

If $X$ is a complex algebraic variety, we let $K^0(X)$ be its Grothendieck group of algebraic vector bundles and $K_0(X)$ be its Grothendieck group of coherent sheaves. For further background on $K$-theory in this context, see \cite{hatcher-Kthy,Karoubi,Weibel}.  When $Y$ is a CW complex, we let $K(Y)$ denote the representable topological complex $K$-theory, that is to say homotopy classes of maps from $Y$ to $BU \times \Z$, where $U(n)$ is the unitary group of $n\times n$ matrices, $BU(n)$ is the classifying space for $U(n)$-bundles, and $BU$ is the colimit of the $BU(n)$; the relation with vector bundles comes from the fact that $BU(n)\simeq\textrm{Gr}_n(\C^\infty)$, the Grassmannian of $n$-planes. For particularly nice spaces, the three groups $K^0(X)$, $K_0(X)$, and $K(X)$ can all be identified with each other; however, in general, they are all distinct and we will need to consider some slightly subtle interactions among them. One of the key properties we use about $K$ is:

\begin{proposition}[{\cite[Corollary 1]{Buhstaber68}}]\label{prop:rmk:different-K-thys}
If $Y$ is the union of an expanding sequence of finite CW-subcomplexes $Y_n$, each with only even-dimensional cells, then the natural map
\begin{equation}\label{eqn:repKequalsinvK}
K(Y)\xrightarrow{\cong}\lim K(Y_n)
\end{equation}
is an isomorphism, where the limit is taken in the category of commutative rings.
\end{proposition}

\begin{remark}
\label{rmk:different-K-thys}
Another definition of ``topological $K$-theory'' one might consider here is $K_{\rm top}(Y)$, the Grothendieck group of topological vector bundles. For compact CW complexes, we have a natural isomorphism $K_{\rm top}(Y) \cong K(Y)$. However, these two notions may differ even for simple examples of infinite CW complexes; see, e.g.,~\cite{JackowskiOliver96}. In fact, the functor $K_{\rm top}$ does not even satisfy Bott periodicity.
\end{remark}

Throughout, we denote the integral cohomology of a CW complex $Y$ by $H^*(Y)$, and for any ring $R$, we use $R_\Q$ to denote $R\otimes_\Z\Q$. For a finite CW complex, we have a \emph{Chern character map}
\[
\ch\colon K(Y)\to H^*(Y)_\Q;
\]
see, e.g., \cite[p.\ 109]{hatcher-Kthy} for background on the Chern character.

\subsection{Characterizing $K(J\C\bP^\infty)\,\widehat{\otimes}_\Z\,\Q$ as quasisymmetric functions}
\label{subsec:characterizingKJCPinfty}

Our goal in this subsection is to compute $K(J\C\bP^\infty)$ in terms of the $K(J_n\C\bP^m)$. We first justify expressing this $K$-theory as a doubly indexed inverse limit. We then use pullback along $q_n$ and the Chern character to identify each finite-stage ring with a ring of truncated quasisymmetric polynomials, and finally pass to the limit.

The first step requires a minor technical extension of \eqref{eqn:repKequalsinvK}, which only allows one to compute $K$-theory as an inverse limit of a \emph{sequence} of CW complexes, as opposed to a doubly infinite sequence. The following result shows that, in our case of interest, \eqref{eqn:repKequalsinvK} remains true for doubly infinite sequences as well.

\begin{lemma}\label{l:repKequalsinvK-double-seq}
Consider the diagram
\[
\xymatrix@C=1.25em@R=1.25em{
\vdots & \vdots &\\
Y_{1,0}\ar@{}[r]|-*[@]{\subset} \ar@{}[u]|-*[@]{\subset} & Y_{1,1} \ar@{}[r]|-*[@]{\subset} \ar@{}[u]|-*[@]{\subset} & \cdots\\
Y_{0,0}\ar@{}[r]|-*[@]{\subset}\ar@{}[u]|-*[@]{\subset} & Y_{0,1} \ar@{}[r]|-*[@]{\subset}\ar@{}[u]|-*[@]{\subset} & \cdots
}
\]
where each $Y_{n,m}$ is a finite CW complex with only even-dimensional cells and all displayed containments are inclusions of subcomplexes. Let $Y=\bigcup_{n,m} Y_{n,m}$ and assume $Y_{n,m}\cap Y_{n',m'}=Y_{\min(n,n'),\min(m,m')}$. Then
\[
K(Y)=\lim_{n,m}K(Y_{n,m})
\]
where the limit is taken in the category of commutative rings.
\end{lemma}
\begin{proof}
We know from \eqref{eqn:repKequalsinvK} that
\[
K(Y)=\lim_r K(Y'_r)
\]
where $Y'_r=\bigcup_{n+m\leq r} Y_{n,m}$. So it remains to prove $K(Y'_r)=\lim_{n+m\leq r} K(Y_{n,m})$. In fact, we show the following stronger result. Let $\cS\subseteq\Z^2_{\geq0}$ be a finite set with the property that if $(n,m)\in\cS$, $n'\leq n$, and $m'\leq m$, then $(n',m')\in\cS$. Letting $Y_\cS=\bigcup_{(n,m)\in\cS}Y_{n,m}$, we prove
\[
K(Y_\cS)=\lim_{(n,m)\in\cS}K(Y_{n,m})
\]
by induction on $|\cS|$. If $|\cS|=1$, then $\cS=\{(0,0)\}$ and there is nothing to prove. For the induction step, fix $\cS$ and $(a,b)\notin\cS$ such that $(a-1,b)\in\cS$ if $a>0$ and $(a,b-1)\in\cS$ if $b>0$. Let $\cS'=\cS\cup\{(a,b)\}$. Since the $Y_{n,m}$ are compact Hausdorff with only even-dimensional cells, the higher $K$-groups are concentrated in even degrees. As a result, long exact sequence on $K$-groups yields
\[
\xymatrix{
0\ar[r] & \widetilde{K}(Y_{\cS'}/Y_{a,b})\ar[r]\ar[d]^-{\cong} & K(Y_{\cS'})\ar[r]\ar[d] & K(Y_{a,b})\ar[r]\ar[d] & 0\\
0\ar[r] & \widetilde{K}(Y_{\cS}/(Y_{\cS}\cap Y_{a,b}))\ar[r] & K(Y_{\cS})\ar[r] & K(Y_{\cS}\cap Y_{a,b})\ar[r] & 0
}
\]
where the two rows are exact and $\widetilde{K}$ denotes the reduced $K$-groups. It follows that the right-hand square of the diagram is cartesian, i.e., $K(Y_{\cS'})$ is the limit of $K(Y_{\cS})$, $K(Y_{a,b})$, and $K(Y_{\cS}\cap Y_{a,b})$ in the category of abelian groups. Since all morphisms in the right-hand square are the natural pullback maps, they are also ring maps, and hence $K(Y_{\cS'})$ is the limit of $K(Y_{\cS})$, $K(Y_{a,b})$, and $K(Y_{\cS}\cap Y_{a,b})$ in the category of rings. Using our hypothesis that $Y_{n,m}\cap Y_{n',m'}=Y_{\min(n,n'),\min(m,m')}$, we see $Y_{\cS}\cap Y_{a,b}=Y_{\cT}$, where $\cT$ is the set of $(n,m)$ with $n\leq a-1$ and $m\leq b$, or $n\leq a$ and $m\leq b-1$. Applying the induction statement to $\cS$ and $\cT$, we have proved the result for $\cS'$.
\end{proof}

We now apply \autoref{l:repKequalsinvK-double-seq} to our case of interest:~the James reduced product of $\C\bP^\infty$.

\begin{corollary}\label{cor:KJCPinfty-double-seq-inverse-limit}
We have
\[
K(J\C\bP^\infty)\ \cong\ \lim_{n,m} K(J_n\bP^m)
\]
and
\[
K(J\C\bP^\infty)\,\widehat{\otimes}_\Z\,\Q\ \cong\ \lim_{n,m}(K(J_n\bP^m)_\Q).
\]
\end{corollary}
\begin{proof}
The first statement follows immediately from \autoref{l:repKequalsinvK-double-seq} upon showing
\begin{equation}\label{intersection-of-JnPms}
(J_n\bP^m)\cap(J_{n'}\bP^{m'})=J_n\bP^{m'},
\end{equation}
where $n\leq n'$ and $m'\leq m$. Every element of $J_n\bP^m$ may be written in the form $(x_1,\dots,x_k)$ where $k\leq n$ and $x_i\in\bP^m\smallsetminus\{e_0\}$; similarly every element of $J_{n'}\bP^{m'}$ may be written in the form $(y_1,\dots,y_\ell)$ where $\ell\leq n'$ and $y_i\in\bP^{m'}\smallsetminus\{e_0\}$. Equality of these two elements implies $k=\ell$ and $x_i=y_i$. Thus, $\ell=k\leq n$ and $x_i=y_i\in\bP^{m'}$, proving \eqref{intersection-of-JnPms}.

To compute $K(J\C\bP^\infty)\,\widehat{\otimes}_\Z\,\Q$, we first note that for $n'\leq n$ and $m'\leq m$, the transition map $K(J_n\bP^m)\to K(J_{n'}\bP^{m'})$ is surjective; indeed, it is true more generally that if $X$ is any compact Hausdorff CW complex with only even-dimensional cells and $A\subset X$ is a subcomplex, then all higher $K$-groups are concentrated in even degrees and so the long exact sequence on $K$-groups shows $K(X)\to K(A)$ is surjective. Since the transition maps are surjective, the induced maps $K(J\C\bP^\infty)\cong\lim_{a,b} K(J_a\bP^b)\to K(J_n\bP^m)$ are surjective. We may therefore write $K(J_n\bP^m)=K(J\C\bP^\infty)/I_{n,m}$ where $I_{n,m}$ is an ideal. Then, by definition, 
\[
K(J\C\bP^\infty)\,\widehat{\otimes}_\Z\,\Q=\lim_{n,m}(K(J\C\bP^\infty)_\Q/(I_{n,m}\otimes_\Z\Q))=\lim_{n,m}(K(J_n\bP^m)_\Q),
\]
where the second equality uses flatness of $\Z\to\Q$.
\end{proof}

Next, let $\pi_i\colon(\bP^m)^n\to\bP^m$ denote projection onto the $i$th factor and let
\begin{equation}\label{eqn:def-of-yi}
y_i=1-[\pi_i^*\cO_{\bP^m}(-1)]
\end{equation}
where $\cO_{\bP^m}(-1)$ is the tautological sub-line bundle on $\bP^m$ and $[\cdot]$ denotes the class in $K$-theory. 
It is well-known (see, e.g., Theorems 2.2 and 4.1 of \cite{Sankaran}) that the map
\begin{equation}\label{eqn:top-K-thy-compared-with-alg-K-thy-Pnm}
K^0((\bP^m)^n)\xrightarrow{\cong}K((\bP^m)^n)
\end{equation}
is an isomorphism and that
\begin{equation}\label{eqn:K0Pmn-integral}
K^0((\bP^m)^n)=\Z[y_1,\dots,y_n]/(y_1^{m+1},\dots,y_n^{m+1}).
\end{equation}

%%%%%%%%%%%%%%%%%%%%%%%%%%%%%%%%%%%%%
%%%%%%%%%%%%%%%%
%%%%%%%%%%%%%%%%		begin comment
%%%%%%%%%%%%%%%%

\begin{comment}
We first show that $K^0(X)_\Q\to K(X)_\Q$ is an isomorphism for many spaces of interest. We will only need apply this for the special case $X=(\bP^m)^n$, where one may check directly that $K^0(X)\to K(X)$ is an isomorphism; however, we record the general result here as it may be of independent interest.

\begin{proposition}
\label{prop:Ktop-equals-Kalg}
Suppose that $X$ is a smooth complex variety with a cellular decomposition in the sense of \cite[Example 1.9.1]{Fulton-int-thy}. Then all maps in the natural commutative diagram
\[
\xymatrix{
K^0(X)_\Q\ar[r]^-{\ch}\ar[d] & A^*(X)_\Q\ar[d]^-{\mathrm{cl}}\\
K(X)_\Q\ar[r]^-{\ch} & H^*(X)_\Q
}
\]
are isomorphisms, where $\textrm{cl}$ denotes the cycle map.
\end{proposition}

\begin{proof}
It suffices to prove that three out of the four maps are isomorphisms. Since $X$ has a cellular decomposition, \cite[Example 19.1.11]{Fulton-int-thy} shows that the cycle map $\textrm{cl}$ is an isomorphism. Since $X$ is a finite CW complex, \cite[Proposition 4.5]{hatcher-Kthy} shows that the topological Chern character $\ch\colon K(X)_\Q\to H^*(X)_\Q$ is an isomorphism. Lastly, it follows from the Grothendieck--Riemann--Roch Theorem that for smooth varieties, the Chern character $\ch\colon K^0(X)_\Q\to A^*(X)_\Q$ is an isomorphism; see e.g., \cite[Example 15.2.16 (b)]{Fulton-int-thy}.
\end{proof}

\end{comment}
%%%%%%%%%%%%%%%%
%%%%%%%%%%%%%%%%		end comment
%%%%%%%%%%%%%%%%
%%%%%%%%%%%%%%%%%%%%%%%%%%%%%%%%%%%%%

The following proposition is the main result of this subsection.

\begin{proposition}
\label{prop:KJmCPn-is-qsym-in-y}
%Let $\pi_i\colon(\bP^m)^n\to\bP^m$ denote projection onto the $i$th factor and let \[y_i=1-[\pi_i^*\cO(-1)]\in K^0((\bP^m)^n)_\Q.\]
With notation as in \eqref{eqn:def-of-yi} and \eqref{eqn:K0Pmn-integral}, the map
\[
q_n^*\colon K(J_n\bP^m)_\Q\to K((\bP^m)^n)_\Q\cong K^0((\bP^m)^n)_\Q=\Q[y_1,\dots,y_n]/(y_1^{m+1},\dots,y_n^{m+1})
\]
is an injection which identifies $K(J_n\bP^m)_\Q$ with quasisymmetric polynomials in $y_1,\dots,y_n$. These injections yield a natural isomorphism
\[
\psi\colon K(J\C\bP^\infty)\,\widehat{\otimes}_\Z\,\Q\xrightarrow{\cong} \mQSymQ.
\]
\end{proposition}

\begin{proof}
Since the Chern character commutes with pullback, we have a commutative diagram
\[
\xymatrix{
K((\bP^m)^n)_\Q\ar[r]_-{\cong}^-{\ch} & H^*((\bP^m)^n)_\Q\\
K(J_n\bP^m)_\Q\ar[u]^-{q_n^*}\ar[r]_-{\cong}^-{\ch} & H^*(J_n\bP^m)_\Q\ar[u]_-{q_n^*}
}
\]
\autoref{james-reduced-cells->monomial-basis} and its proof show that, on cohomology, $q_n^*$ is injective and identifies $H^*(J_n\bP^m)_\Q$ with the subring of $H^*((\bP^m)^n)_\Q=\Q[x_1,\dots,x_n]/(x_1^{m+1},\dots,x_n^{m+1})$ consisting of quasisymmetric polynomials in the $x_i$. As a result, the $K$-theoretic pullback $q_n^*$ is injective and identifies $K(J_n\bP^m)_\Q$ with the subring of $K((\bP^m)^n)_\Q$ consisting of polynomials $f(y_1,\dots,y_n)$ for which $\ch(f)$ is quasisymmetric in the $x_i$.

Thus, given $f(y_1,\dots,y_n)\in K((\bP^m)^n)_\Q$, we must show the following two conditions are equivalent:
\begin{itemize}
	\item[(i)] $f$ is quasisymmetric in the $y_i$, and 
	\item[(ii)] $\ch(f)$ is quasisymmetric in the $x_i$.
\end{itemize}

To prove (i) implies (ii), first suppose $f(y_1,\dots,y_n)$ is a quasisymmetric power series. Let $g(x) \in \mathbb{Q} \llbracket x \rrbracket$ be any power series with no constant term. Then it is immediate from the definition of quasisymmetry that the composition $f(g(x_1),\dots,g(x_n))$ is also quasisymmetric.

Now, consider $f(y_1,\dots,y_n)\in K((\bP^m)^n)_\Q$; we show that $\ch(f)$ is quasisymmetric in the $x_i$. 
%It is enough to do so when $f$ is a monomial quasisymmetric polynomial, since such polynomials form a basis and the Chern character is linear. Let $\alpha_1,\dots, \alpha_k$ be a sequence of positive integers with $k\leq n$. Then letting
%\[
%f=\sum_{i_1<\dots<i_k} y_{i_1}^{\alpha_1}\dots y_{i_k}^{\alpha_k},
%\]
We have
\[
\ch(f)=f(1-e^{-x_1},\dots,1-e^{-x_n});
\]
this is because $x_i$ is the class of $\bP^m\times\dots\times\bP^m\times\bP^{m-1}\times\bP^m\times\dots\times\bP^m$ with $\bP^{m-1}$ appearing in the $i$-th factor, and so $-x_i=c_1(\pi_i^*\cO(-1))$. Since $1-e^{-x}$ is a power series with no constant term, it follows that $\ch(f)$ is quasisymmetric, establishing that (i) implies (ii).

%Hence,
%\begin{align*}
%\ch(f) &=\sum_{i_1<\dots<i_k} (1-e^{-x_{i_1}})^{a_1}\dots (1-e^{-x_{i_k}})^{a_k}\\
%&=\sum_{0\leq \ell_j\leq a_j}(-1)^{\ell_1+\dots+\ell_k}\prod_{j=1}^k {a_j\choose \ell_j} \sum_{i_1<\dots<i_k} e^{-\ell_1 x_{i_1}}\dots e^{-\ell_k x_{i_k}}.
%\end{align*}
%Expanding, we have
%\[
%\sum_{i_1<\dots<i_k} e^{-\ell_1 x_{i_1}}\dots e^{-\ell_k x_{i_k}} = \sum_{0\leq d_j\leq m} \frac{(-1)^{d_1+\dots+d_k}}{d_1!\dots d_k!} \sum_{i_1<\dots<i_k} x_{i_1}^{d_1}\dots x_{i_k}^{d_k},
%\]
%which is a sum of monomial quasisymmetric polynomials. Therefore, $\ch(f)$ is quasisymmetric, establishing that (i) implies (ii).

Since (i) implies (ii), we see that the quasisymmetric polynomials in $y_i$ are contained in $q_n^*K(J_n\bP^m)_\Q$, the set of polynomials $f$ such that $\ch(f)$ is quasisymmetric.
 However, these two sets have the same dimension when viewed as $\Q$-vector spaces, since $\ch$ defines an isomorphism between $K(J_n\bP^m)_\Q$ and the vector space of quasisymmetric polynomials in the $x_i$. Thus, the sets coincide and (ii) implies (i), as well.

Lastly, having now shown that $K(J_n\bP^m)_\Q$ agrees with the quasisymmetric polynomials in the ring $\Q[y_1,\dots,y_n]/(y_1^{m+1},\dots,y_n^{m+1})$, we see from \autoref{cor:KJCPinfty-double-seq-inverse-limit} that
\[
K(J\C\bP^\infty)\,\widehat{\otimes}_\Z\,\Q\ \cong\ \lim_{n,m} (K(J_n\bP^m)_\Q)\ \cong\ \mQSymQ,
\]
where the limit is in the category of commutative rings.
\end{proof}

We end this subsection with the following standard computation, which we make use of later. We include a full proof for completeness.

\begin{lemma}[{cf.~\cite[Corollaries 2.8 and 2.11 of Chapter 4, \S2]{Karoubi}}]
\label{l:KclassProdPn}
For $0\leq r\leq m$, in $K^0(\bP^m)$, we have
\[
[\cO_{\bP^{r}}]=(1-[\cO_{\bP^m}(-1)])^{m-r}.
\]
Let $\pi_i\colon(\bP^m)^n\to\bP^m$ denote projection onto the $i$th factor and let \[
y_i=1-[\pi_i^*\cO_{\bP^m}(-1)]\in K^0((\bP^m)^n).
\]
Then, for $0\leq r_1,\dots,r_n\leq m$, in $K^0((\bP^m)^n)$, we have
\[
[\cO_{\bP^{r_1}\times\dots\times\bP^{r_n}}]=\prod_{i=1}^n y_i^{m-r_i}.
\]
\end{lemma}
\begin{proof}
The second statement follows immediately from the first; alternatively, see the citation to \cite{Karoubi}. To prove the first statement, we use the exact sequence
\[
0\to\cO_{\bP^{m}}(-1)\to\cO_{\bP^{m}}\to\cO_{\bP^{m-1}}\to0,
\]
which shows
\[
[\cO_{\bP^{m-1}}]=1-[\cO_{\bP^m}(-1)].
\]
Lastly, observe that $[\cO_{\bP^{r}}]=[\cO_{\bP^{m-1}}]^{m-r}$ since $m-r$ generic hyperplanes intersect transversely.
\end{proof}

\subsection{$K$-classes are represented by quasisymmetric monomial glides}

Our first aim in this subsection is to associate a $K$-class to each cell $e_\alpha$ of $J\C\bP^\infty$.
Let $\alpha=(\alpha_1,\dots, \alpha_k)$. Then for any $n\geq k$ and $m\geq \max_i \alpha_i$, we see $e_\alpha$ is a cell of $J_n\bP^m$.
 If $J_n\bP^m$ were a smooth complex variety, we would want to associate $e_\alpha$ with the class of its structure sheaf. However, as discussed in \autoref{sec:not_normal}, this approach is not viable, so we must substitute a more delicate construction.
 Although we cannot use the class of the structure sheaf on $J_n\bP^m$ itself, we can instead pull $e_\alpha$ back to the smooth variety $(\bP^m)^n$ via the quotient map $q_n\colon(\bP^m)^n\to J_n\bP^m$, and then consider the associated $K$-class of the closure of $q_n^{-1}(e_\alpha)$. It turns out that we get better behaviour under the necessary limits if we first take the Poincar\'e dual of the closure of $q_n^{-1}(e_\alpha)$ before taking the $K$-class; the reason for this is that we want the codimension of the $K$-class to be independent of the choice of $n$ and $m$. %to compare with the cohomology class of $e_\alpha$,

We can describe $\overline{q_n^{-1}(e_\alpha)}$ and its Poincar\'e dual explicitly as follows. Let $\cI_n$ be the set of order-preserving injections $\iota\colon [k] \to [n]$. For every $\iota\in\cI_n$, define $b(\iota)=(b_1,\dots,b_n)$ by 
\[
b_i = \begin{cases}
	\alpha_j, & \text{if } i = \iota(j); \\
	0, & \text{if } i \notin \im \iota
\end{cases}
\]
and let $r(\iota)=(r_1,\dots,r_n)$ be defined by $r_i=m-b(\iota)_i$. We set $\cB_{\alpha,n}=\{b(\iota)\mid\iota\in\cI_n\}$ and $\cR_{\alpha,n,m}=\{r(\iota)\mid\iota\in\cI_n\}$. Then
\[
\overline{q_n^{-1}(e_\alpha)}=\bigcup_{(b_1,\dots,b_n)\in\cB_{\alpha,n}}\bP^{b_1}\times\dots\times\bP^{b_n}
\]
and its Poincar\'e dual is given by 
\[
Z_{\alpha,n,m}=\bigcup_{(r_1,\dots,r_n)\in\cR_{\alpha,n,m}} \bP^{r_1}\times\dots\times\bP^{r_n}.
\]
We consider the class $[\cO_{Z_{\alpha,n,m}}]\in K^0((\bP^m)^n)_\Q$, where we use smoothness to resolve $\cO_{Z_{\alpha,n,m}}$ by locally-free sheaves to produce a class in $K^0$. Using then the natural isomorphism \eqref{eqn:top-K-thy-compared-with-alg-K-thy-Pnm}, we conflate this algebraic class with the topological class, which we also write as $[\cO_{Z_{\alpha,n,m}}]\in K((\bP^m)^n)_\Q$.

\begin{lemma}\label{lemma:unique_classes}
    For any cell $e_\alpha$ of $J_n\bP^m$, there is at most one class $[e_\alpha]_{n,m}\in K(J_n\bP^m)_\Q$ that pulls back to $[\cO_{Z_{\alpha,n,m}}]\in K((\bP^m)^n)_\Q$ under the quotient map $q_n$. Similarly, for any cell $e_\alpha$ of $J\C\bP^\infty$, there is at most one class \[
[e_\alpha]\in K(J\C\bP^\infty)\,\widehat{\otimes}_\Z\,\Q
\]
which pulls back to $[e_\alpha]_{n,m}$ for all $n,m$ sufficiently large.
\end{lemma}
\begin{proof}
	The uniqueness of the class $[e_\alpha]_{n,m}\in K(J_n\bP^m)_\Q$ is immediate from the injectivity statement of \autoref{prop:KJmCPn-is-qsym-in-y}. The uniqueness of $
[e_\alpha]\in K(J\C\bP^\infty)\,\widehat{\otimes}_\Z\,\Q
$ is then immediate from the definition of inverse limits.
\end{proof}

%     The existence of $[e_\alpha]_{n,m}\in K(J_n\bP^m)_\Q$ will be established below in \autoref{thm:K-classesAgreewithGlidePolys}, which also then establishes the existence and uniqueness of the class $[e_\alpha]\in K(J\C\bP^\infty)\,\widehat{\otimes}_\Z\,\Q$.

% \begin{definition}\label{def:K-class-ealpha-n-m}
% We define a \newword{$K$-class of $e_\alpha$} viewed as a cell of $J_n\bP^m$ to be a class $[e_\alpha]_{n,m}\in K(J_n\bP^m)_\Q$ that pulls back to $[\cO_{Z_{\alpha,n,m}}]\in K((\bP^m)^n)_\Q$ under the quotient map $q_n$.

% We define a \newword{$K$-class of $e_\alpha$} viewed as a cell of $J\C\bP^\infty$ to be a class
% \[
% [e_\alpha]\in K(J\C\bP^\infty)\,\widehat{\otimes}_\Z\,\Q
% \]
% which pulls back to $[e_\alpha]_{n,m}$ for all $n,m$ sufficiently large.
% \end{definition}

A \newword{Schauder basis} of a topological $\Q$-vector space $V$ is a sequence $b_1, b_2, \dots$ of vectors such that any $v \in V$ can be uniquely written as a convergent series $v = \sum_{i=1}^\infty a_i b_i$ with $a_i \in \Q$.
The following theorem shows that the classes described in \autoref{lemma:unique_classes} always exist and also establishes \autoref{thm:main-K}. In light of this result, we refer to $[e_\alpha]_{n,m}\in K(J_n\bP^m)_\Q$ and $[e_\alpha]\in K(J\C\bP^\infty)\,\widehat{\otimes}_\Z\,\Q$ as the \newword{$K$-classes} of the cell $e_\alpha$.

\begin{theorem}
\label{thm:K-classesAgreewithGlidePolys}
Let $\alpha=(\alpha_1,\dots, \alpha_k)$ be a composition. Let $n\geq k$ and $m\geq \max_i \alpha_i$ so that $e_\alpha$ is a cell of $J_n\bP^m$. Then the following hold.
\begin{enumerate}
\item\label{K-classesAgreewithGlidePolys::eKnm} $[\cO_{Z_{\alpha,n,m}}]\in K((\bP^m)^n)_\Q$ is in the image of $q_n^*\colon K(J_n\bP^m)_\Q\to K((\bP^m)^n)_\Q$, so determines a unique class
\[
[e_\alpha]_{n,m}\in K(J_n\bP^m)_\Q.
\]
%which we refer to as the $K$-class of $e_\alpha$ in $K(J_n\bP^m)$.

\item\label{K-classesAgreewithGlidePolys::eK} If $N\geq n$ and $M\geq m$, then under the inclusion $\iota\colon J_n\bP^m\to J_N\bP^M$, we have
\[
[e_\alpha]_{n,m}=\iota^*[e_\alpha]_{N,M}.
\]
This system of equalities therefore determines a class
\[
[e_\alpha]\in \lim_{N,M}(K(J_N\bP^M)_\Q)=K(J\C\bP^\infty)\,\widehat{\otimes}_\Z\,\Q.
\]

\item\label{K-classesAgreewithGlidePolys::basis} The $[e_\alpha]$ form a Schauder basis for $K(J\C\bP^\infty)\,\widehat{\otimes}_\Z\,\Q$.

\item\label{K-classesAgreewithGlidePolys::main} Under the natural isomorphism 
\[
\psi\colon K(J\C\bP^\infty)\,\widehat{\otimes}_\Z\,\Q\xrightarrow{\cong} \mQSymQ
\]
from \autoref{prop:KJmCPn-is-qsym-in-y}, we have
\[
\psi([e_\alpha])=\MK_\alpha(y_1,y_2,\dots).
\]
\end{enumerate}
\end{theorem}

\begin{proof}
Recall that under the isomorphism $\psi$, the element $y_i$ corresponds to $1-[\pi_i^*\cO(-1)]\in K((\bP^M)^N)_\Q$ for any $N\geq i$, where $\pi_i\colon(\bP^M)^N\to\bP^M$ is the projection map onto the $i$th factor. Thus, we will henceforth conflate $y_i$ and $1-[\pi_i^*\cO(-1)]$. To prove (\ref{K-classesAgreewithGlidePolys::eKnm}), (\ref{K-classesAgreewithGlidePolys::eK}), and (\ref{K-classesAgreewithGlidePolys::main}), we claim it is enough to show
\begin{equation}\label{eqn:main-n-m}
[\cO_{Z_{\alpha,n,m}}]=\MK_\alpha(y_1,\dots,y_n,0,0,\dots).
\end{equation}
Indeed, since $\MK_\alpha$ is quasisymmetric by \autoref{thm:two-polys-agree}, \eqref{eqn:main-n-m} implies that $[\cO_{Z_{\alpha,n,m}}]$ is quasisymmetric, and hence lives in $K(J_n\bP^m)_\Q$ by \autoref{prop:KJmCPn-is-qsym-in-y}; this proves the existence part of (\ref{K-classesAgreewithGlidePolys::eKnm}). The uniqueness part of (\ref{K-classesAgreewithGlidePolys::eKnm}) is immediate from the injectivity statement of \autoref{prop:KJmCPn-is-qsym-in-y}. Next, the map $\iota^*\colon K(J_N\bP^M)_\Q\to K(J_n\bP^m)_\Q$ corresponds to the map 
\[
\Q[y_1,\dots,y_N]/(y_i^{M+1})\cong K((\bP^M)^N)_\Q\to K((\bP^m)^n)_\Q\cong \Q[y_1,\dots,y_n]/(y_i^{m+1})
\]
sending $y_i$ to $0$ for $i>n$. It follows from \eqref{eqn:main-n-m} that
\begin{equation}\label{eqn:main-n-m-helper}
\iota^*[\cO_{Z_{\alpha,N,M}}]=\iota^*\MK_\alpha(y_1,\dots,y_N,0,0,\dots)=\MK_\alpha(y_1,\dots,y_n,0,0,\dots)=[\cO_{Z_{\alpha,n,m}}],
\end{equation}
proving (\ref{K-classesAgreewithGlidePolys::eK}). Since $\psi$ is constructed from isomorphisms at each finite level 
\[
K(J_n\bP^m)_\Q\ \cong\ \mQSymQ/\bigl(\mQSymQ\cap(y_1^{m+1},\dots,y_n^{m+1},y_{n+1},y_{n+2},\dots)\bigr),
\]
where the intersection is taken in $\Q\llbracket y_1,y_2,\dots\rrbracket$,
we see that (\ref{K-classesAgreewithGlidePolys::main}) follows from \eqref{eqn:main-n-m} and \eqref{eqn:main-n-m-helper}.

We turn now to the proof of \eqref{eqn:main-n-m}. For $\sigma=(\sigma_1,\dots,\sigma_n)\in\cP_\alpha$, write
\[
W_\sigma=\bP^{m-\sigma_1}\times\dots\times\bP^{m-\sigma_n}.
\]
The irreducible components of $Z_{\alpha,n,m}$ are the varieties $W_\tau$ for $\tau\in\mathcal S_\alpha$. Indeed, $W_\sigma\subseteq W_\tau$ if and only if $\sigma\geq\tau$ componentwise, while two elements of $\mathcal S_\alpha$ have the same sum and hence are comparable only if they are equal. Moreover, scheme-theoretically,
\[
W_\sigma\cap W_\tau=W_{\sigma\vee\tau},
\]
where $\vee$ denotes componentwise maximum. Hence, in the notation of \cite[Theorem~1]{Allen}, the minimal collection $\cP'$ obtained from the irreducible components of $Z_{\alpha,n,m}$ by taking intersections and irreducible components is
\[
\cP'=\{W_\sigma:\sigma\in\cP_\alpha\}.
\]

The reducedness hypothesis of \cite[Theorem~1]{Allen} also holds. In standard multihomogeneous coordinates, the ideal of each $W_\sigma$ is generated by variables. Intersections of such ideals, and sums of such intersections with another ideal generated by variables, are squarefree monomial ideals. Thus $Z_{\alpha,n,m}$ and every scheme of the form $W_\tau\cap\bigcup_i W_{\sigma_i}$ are reduced. Applying a theorem of A.~Knutson \cite[Theorem~1]{Allen} to $Z_{\alpha,n,m}$ and pushing forward along $Z_{\alpha,n,m}\hookrightarrow(\bP^m)^n$, we obtain
\[
[\cO_{Z_{\alpha,n,m}}]=\sum_{W\in\cP'}\mu_{\cP'}(W)[\cO_W].
\]
(Knutson \cite[Theorem~3]{Allen} combines this M\"obius inversion formula with Brion's degeneration theorem for irreducible multiplicity-free subvarieties \cite{Brion}; in our context, no degeneration is needed because $Z_{\alpha,n,m}$ is already a union of Schubert varieties.)
Here, $\cP'$ is ordered by inclusion and $\mu_{\cP'}$ is characterized by
\[
\sum_{W'\geq W}\mu_{\cP'}(W')=1.
\]
The map $\sigma\mapsto W_\sigma$ is an order-reversing bijection from $\cP_\alpha$ to $\cP'$, so this recurrence becomes \eqref{eq:mobius}; in particular,
\[
\mu_{\cP'}(W_\sigma)=\mu_{\cP_\alpha}(\sigma).
\]
(Technically, after pushing forward, Knutson's formula lies in $K_0((\bP^m)^n)$ rather than $K^0((\bP^m)^n)$, but $(\bP^m)^n$ is smooth, so these groups are naturally isomorphic.) As a result,
\begin{align*}
[\cO_{Z_{\alpha,n,m}}] &=\sum_{\sigma=(\sigma_1,\dots,\sigma_n)\in\cP_\alpha}\mu_{\cP_\alpha}(\sigma)[\cO_{\bP^{m-\sigma_1}\times\dots\times\bP^{m-\sigma_n}}]\\
&=\sum_{\sigma=(\sigma_1,\dots,\sigma_n)\in\cP_\alpha}\mu_{\cP_\alpha}(\sigma)y_1^{\sigma_1}\dots y_n^{\sigma_n}\\
&=\MK_\alpha(y_1, \dots, y_n);
\end{align*}
the second equality follows from \autoref{l:KclassProdPn} and the last equality holds by definition. We have therefore proved \eqref{eqn:main-n-m}.

Lastly, to show (\ref{K-classesAgreewithGlidePolys::basis}), we prove that the $\MK_\alpha$ form a Schauder basis for $\mQSymQ$, the ring of quasisymmetric power series with rational coefficients. The key observation here is that the lowest degree terms of $\MK_\alpha$ are the monomial quasisymmetric function $M_\alpha$, so linear independence is immediate. Consider $f \in \mQSymQ$ with lowest degree terms in degree $\ell$; let $f_\ell$ denote the degree $\ell$ homogeneous part of $f$. Since the degree $\ell$ homogeneous part of $\mQSymQ$ is finite-dimensional, with a basis of monomial quasisymmetric functions $M_\beta$ such that $|\beta| = \ell$, we may write $f_\ell$ as a finite $\Q$-linear combination of monomial quasisymmetric functions:
\[
f_\ell = \sum_{i = 1}^k c_i M_{\beta_i},
\] 
for some compositions $\beta_i$ and coefficients $c_i \in \Q$. Now, consider 
\[
f' = f - \sum_{i = 1}^k c_i \MK_{\beta_i} \in \mQSymQ.
\]
The lowest degree terms of $f'$ are in degree $\ell' > \ell$, so by iterating this process we may write $f$ as a countable sum of quasisymmetric monomial glides $\MK_\alpha$ such that there are only finitely many $\alpha$ with $|\alpha| < p$ appearing for each degree $p$.
\end{proof}

\section{The James reduced product is not normal}\label{sec:not_normal}

One may wonder whether we cannot endow $J_n\bP^m$ with the structure of an algebraic variety and define the $K$-class of a cell $e_\alpha$ as the structure sheaf of the closure of the cell, i.e.~$[\cO_{\overline{e}_\alpha}]\in K_0(J_n\bP^m)$. 
In our sequel \cite{Pechenik.Satriano:double}, we prove $J_n\bP^m$ can indeed be given the structure of an algebraic variety; however, this variety is quite singular.
\autoref{prop:JnPm-not-variety} shows that any variety structure on $J_n\bP^m$ must necessarily be non-normal, hence singular. Consequently, there would be no reason to expect $K_0(J_n\bP^m)$ to be isomorphic to $K^0(J_n\bP^m)$, and no reason to expect $[\cO_{\overline{e}_\alpha}] \in K_0(J_n\bP^m)$ to be written as a linear combination of vector bundles. Therefore, there is no natural way to identify it with a class in $K^0((\bP^m)^n)$, which we required in the proof of our main results.

\begin{theorem}
\label{prop:JnPm-not-variety}
Let $m>0$ and $n>1$. Suppose $J_n\bP^m$ can be given the structure of a complex variety such that $q_n\colon(\bP^m)^n\to J_n\bP^m$ is a map of varieties. Then $J_n\bP^m$ is non-normal and $q_n$ is the normalization map.
\end{theorem}
\begin{proof}
Since $q_n$ is surjective and $(\bP^m)^n$ is irreducible, $J_n\bP^m$ is irreducible as well. Since $J_n\bP^m$ is a variety, it is separated; it follows that $q_n$ is proper as $(\bP^m)^n$ is proper. Since, by construction, $q_n$ is also quasi-finite, Zariski's Main Theorem tells us $q_n$ is finite. 

There is a unique top-dimensional cell $e_{(m,\dots,m)}$ of $J_n\bP^m$. Its complement is equal to $q_n(Z)$, where $Z=(\bP^m)^n\smallsetminus(\bA^m)^n$. Since $Z$ is closed and $q_n$ is a closed map (because it is proper), we see $e_{(m,\dots,m)}$ is Zariski open in $J_n\bP^m$. By generic flatness, there is a non-empty Zariski open subset $U\subset J_n\bP^m$ over which $q_n$ is flat. Since $J_n\bP^m$ is irreducible, $V\coloneqq U\cap e_{(m,\dots,m)}$ is non-empty and dense. Over $V$, the map $q_n$ is flat and bijective, hence an isomorphism. Therefore, $q_n$ is a finite birational map.

Let $\nu\colon\widetilde{J_n\bP^m}\to J_n\bP^m$ be the normalization map. Since $(\bP^m)^n$ is normal and $q_n$ is surjective, by the universal property of normalization, there is a unique map $\widetilde{q}_n$ making the diagram
\[
\xymatrix{
(\bP^m)^n\ar[rr]^-{\widetilde{q}_n}\ar[dr]_-{q_n} & & \widetilde{J_n\bP^m}\ar[dl]^-{\nu}\\
& J_n\bP^m &
}
\]
commute. Since $\nu$ and $q_n$ are finite, $\widetilde{q}_n$ is as well. Since $q_n$ is an isomorphism over $V\subset J_n\bP^m$ and $(\bP^m)^n$ is smooth, we see $V$ is smooth, hence normal. As a result, $\nu$ is an isomorphism over $V$ and so $\widetilde{q}_n$ is a finite birational map to a normal variety. Another application of Zariski's Main Theorem tells us $\widetilde{q}_n$ is an isomorphism. In other words, $q_n$ is the normalization map of $J_n\bP^m$.

Lastly, if $J_n\bP^m$ were normal, then $q_n$ would be an isomorphism. This is not possible since $q_n$ is $n$-to-$1$ over the cell $e_{(1)}$ and $n>1$.
\end{proof}

It would be interesting to find an embedding of $J(\C\bP^\infty)$ inside a smooth infinite-type scheme $W$ of the same homotopy type. If $W$ had a stratification by complex affine spaces, it would be a sort of ``James reduced product analogue'' of a \emph{thick Kashiwara flag variety} \cite{Kashiwara} and come with another canonical basis of $K$-theory in analogy with \cite{Kashiwara.Shimozono, Lam.Schilling.Shimozono}.

\section{Extending to flag varieties}\label{sec:flag}

We end this paper by generalizing \autoref{thm:main-H} to a much larger class of CW complexes, particularly with applications to James reduced products of generalized flag varieties. We suppress some of the details that are identical to those in the $\C\bP^\infty$ case. We first describe the cohomology of $J(X)$ for any CW complex $X$ satisfying mild finiteness and parity hypotheses. We then specialize this description to finite-dimensional generalized flag varieties and, finally, to the classifying spaces $BU(k)$, where we obtain a positive combinatorial rule for the cellular structure coefficients of $H^*(J(BU(k)))$.

\subsection{James reduced products of CW complexes}\label{subsec:general-CW-HJX}

Let $X$ denote a CW complex with only even-dimensional cells, only finitely many cells of any given dimension, and a unique $0$-cell.
We denote by $\{e_\theta : \theta \in I\}$ the cells of $X$, where $\theta$ runs through some index set $I$; we require that $I$ contains an element $0$ such that $e_0$ is the unique $0$-cell. %Let $d(a)=\dim e_a$. 
We then have 
\[
H_*(X; \Z) \cong \bigoplus_{\theta \in I} \Z e_\theta \quad\textrm{and}\quad H^*(X; \Z) \cong \bigoplus_{\theta \in I} \Z x_\theta,
\]
where $x_\theta$ denotes the function dual to the cell $e_\theta$.

A \newword{weak $I$-composition} is a finite sequence of elements of $I$, and an \newword{$I$-composition} is a finite sequence of elements of $I\smallsetminus\{0\}$. Every weak $I$-composition $T$ has an associated \newword{positive part} $T^+$ obtained from $T$ by deleting the $0$ terms. For example, if $\theta_1,\theta_2,\theta_3,\theta_4\in I\smallsetminus\{0\}$, then the positive part of $(\theta_1,0,\theta_2,\theta_3,0,0,\theta_4)$ is $(\theta_1,\theta_2,\theta_3,\theta_4)$. Note that \emph{compositions} and \emph{weak compositions} as discussed in \autoref{sec:cohomology} are, in this more general terminology, \emph{$\Z_{\geq 0}$-compositions} and \emph{weak $\Z_{\geq 0}$-compositions}, respectively.

The CW structure on the James reduced product $J_n(X)$ is induced from the cellular quotient map
\[
q_n\colon X^n\to J_n(X),
\]
which, by construction, identifies the cells $e_{\theta_1}\times\dots\times e_{\theta_n}$ and $e_{\kappa_1}\times\dots\times e_{\kappa_n}$ if and only if \[
(\theta_1,\dots,\theta_n)^+=(\kappa_1,\dots,\kappa_n)^+.
\]
 Hence, the cells of $J_n(X)$ are indexed by the $I$-compositions $\Theta = (\theta_1,\dots,\theta_k)$ where $k\leq n$. We denote the corresponding cell of $J_n(X)$ by $e_{(\theta_1,\dots, \theta_k)}$. We see then that the cells $e_\Theta$ of $J(X)$ are indexed by all $I$-compositions $\Theta$ of arbitrary length. As was the case for $J(\C\bP^\infty)$, we have
\[
H_*(J(X); \Z)\cong \bigoplus_\Theta \Z e_\Theta \quad\textrm{and}\quad H^*(J(X); \Z)\cong \bigoplus_\Theta \Z x_\Theta,
\]
where $x_\Theta$ denotes the function dual to the cell $e_\Theta$ and both sums are over all $I$-compositions $\Theta$. We therefore again have
\begin{equation*}%\label{eqn:H-is-inverse-limit-james-in-general}
H^*(J(X); \Z)=\lim_{n\to\infty}H^*(J_n(X);\Z).
\end{equation*}

We prove that $H^*(J(X); \Z)$ is built out of $\QSym$, where one substitutes cohomology classes from $H^*(X; \Z)$ in for the variables $x_i$. In order to make this precise, we first introduce a new ring as follows.

\begin{definition}
\label{def:QSym-substitute}
Let $R$ be a commutative graded ring which is free as a graded $\Z$-module. Let $\eta\colon R\to\Z$ be a map of graded rings where $\Z$ is placed in degree $0$. For every $n\geq1$, $1\leq i\leq n$, and $r\in R$, let
\[
r^{(i)} \coloneqq 1\otimes\dots\otimes1\otimes r\otimes1\otimes\dots\otimes1 \in  R^{\otimes n}
\]
with $r$ in the $i$th tensor factor. Let
\[
\sQSym_n(R)\subset R^{\otimes n}
\]
be the graded subring generated by expressions of the form
\[
M_{n,(r_1,\dots,r_k)}\coloneqq \sum_{1\leq i_1<\dots<i_k\leq n} r_1^{(i_1)}r_2^{(i_2)}\cdots r_k^{(i_k)}
\]
where $r_j\in\ker\eta$. The map $\eta$ yields a graded map $R^{\otimes(n+1)}\to R^{\otimes n}\otimes_{\Z} \Z= R^{\otimes n}$ for each $n$, which induces a graded map
\[
\sQSym_{n+1}(R)\to\sQSym_n(R).
\]
This map sends $M_{n+1,(r_1,\dots,r_k)}$ to $M_{n,(r_1,\dots,r_k)}$.
We let
\[
\sQSym(R)=\lim_n\sQSym_n(R)
\]
be the limit in the category of graded rings. Furthermore, for every choice $(r_1,\dots,r_k)$ with $r_j\in\ker\eta$, we have an element
\[
M_{(r_1,\dots,r_k)}\in \sQSym(R)
\]
mapping to $M_{n,(r_1,\dots,r_k)}$ for each $n\geq k$.
\end{definition}

In our case of interest, we take $R=H^*(X;\Z)$ and the map $\eta\colon H^*(X;\Z)\to H^*(e_0;\Z)=\Z$ in \autoref{def:QSym-substitute} is taken to be the pullback map induced by the inclusion $e_0\in X$. For $\theta\neq0$, the class $x_\theta$ has positive degree and hence belongs to $\ker\eta$. For notational convenience, for each $I$-composition $\Theta=(\theta_1,\dots, \theta_k)$, we let
\[
M_{n,\Theta}\coloneqq M_{n,(x_{\theta_1},\dots,x_{\theta_k})}\in \sQSym_n(H^*(X;\Z))
\]
and 
\[
M_\Theta\coloneqq M_{(x_{\theta_1},\dots,x_{\theta_k})}\in \sQSym(H^*(X;\Z));
\]
recall that $x_{\theta_i}\in H^*(X;\Z)$ is the dual function to the cell $e_{\theta_i}$. %We refer to $M_\alpha$ as the \emph{quasisymmetric monomial function associated to} $\alpha$. 
Note that for each $n$, the set
\[
\left\{ M_{n,\Theta} : \Theta=(\theta_1,\dots,\theta_k),k\leq n\right\}
\]
is a homogeneous basis for $\sQSym_n(H^*(X; \Z))$.

\begin{theorem}\label{thm:cohomology-of-James-reduced-product-in-general}
Let $X$ be a CW complex with only even-dimensional cells, only finitely many cells of any given dimension, and a unique $0$-cell $e_0$. Then we have a natural isomorphism of graded rings
\[
H^*(J(X);\Z)\cong\sQSym(H^*(X;\Z))
\]
sending $x_\Theta$ to $M_\Theta$ for every $I$-composition $\Theta$.
\end{theorem}
\begin{proof}
Since we know that $H^d(X;\Z)$ is free of finite rank for every $d$, the K\"unneth formula (see, e.g., \cite[Theorem 3.16]{hatcher-at}) tells us that
\begin{equation}\label{eqn:kunneth-X-power}
H^*(X^n;\Z)\cong H^*(X;\Z)^{\otimes n}.
\end{equation}

For every $I$-composition $\Theta=(\theta_1,\dots, \theta_k)$, any $n\geq k$, and any order-preserving injection $\iota\colon [k] \to [n]$, define a cell $e_\iota$ of $X^n$ by
\[
e_\iota\coloneqq e_{\kappa_1}\times\dots\times e_{\kappa_n},
\]
where 
\[
\kappa_j \coloneqq  \begin{cases}
	0, & \text{if } j \notin \im \iota; \\
	\theta_{\iota^{-1}(j)}, & \text{otherwise.}
\end{cases}
\]
We see that
\[
q_n^{-1}(e_\Theta)=\coprod_\iota e_\iota,
\]
where the disjoint union is over all order-preserving injections $\iota :[k] \to [n]$.
As a result, making use of Equation~\eqref{eqn:kunneth-X-power} and the notation introduced in \autoref{def:QSym-substitute}, we have
\[
q_n^*(x_\Theta)=\sum_{1\leq i_1<\dots<i_k\leq n} x_{\theta_1}^{(i_1)}\dots x_{\theta_k}^{(i_k)}=M_{n,\Theta}\in\sQSym_n(H^*(X;\Z)).
\]
For each $n$, the set $\{x_\Theta : \Theta=(\theta_1,\dots,\theta_k),k\leq n\}$ is a homogeneous basis for $H^*(J_n(X);\Z)$ and the set $\{M_{n,\Theta} : \Theta=(\theta_1,\dots,\theta_k),k\leq n\}$ is a homogeneous basis for $\sQSym_n(H^*(X;\Z))$. It follows that $q_n^*$ induces a graded isomorphism
\[
q_n^* : H^*(J_nX;\Z)\xrightarrow{\cong}\sQSym_n(H^*(X;\Z)).
\]

To finish the proof of the theorem, we note that the commutative diagram
\[
\xymatrix{
X^n\ar[r]^{\jmath'_n}\ar[d]^-{q_n} & X^{n+1}\ar[d]^-{q_{n+1}}\\
J_nX\ar[r]^{\jmath_n} & J_{n+1}X
}
\]
induces a commutative diagram
\[
\xymatrix{
H^*(J_{n+1}X;\Z)\ar[r]^-{q^*_{n+1}}_-{\cong} \ar[d]^-{\jmath^*_n} & \sQSym_{n+1}(H^*(X;\Z))\ar[d]^-{\pi_n}\ar@{^{(}->}[r] & H^*(X;\Z)^{\otimes(n+1)}\ar[r]_-{\cong}\ar[d] & H^*(X^{n+1};\Z)\ar[d]_-{(\jmath'_n)^*}\\
H^*(J_nX;\Z)\ar[r]^-{q^*_n}_-{\cong} & \sQSym_n(H^*(X;\Z))\ar@{^{(}->}[r] & H^*(X;\Z)^{\otimes n}\ar[r]_-{\cong} & H^*(X^n;\Z)
}
\]
of graded rings. Hence,
\[
H^*(J(X);\Z)=\lim_{n\to\infty} H^*(J_n(X);\Z)\cong \lim_{n\to\infty}\sQSym_n(H^*(X;\Z))=\sQSym(H^*(X;\Z)),
\]
where the limits are taken in the category of graded rings.
\end{proof}

\subsection{James reduced products of generalized flag varieties}
\label{subsec:generalized-flag-vars-HJX}

%The most interesting 
We obtain particularly interesting applications of \autoref{thm:cohomology-of-James-reduced-product-in-general} %are probably 
in the case where 
$X$ is a generalized flag variety, as in many of these cases we have significant combinatorial understanding of the cohomology ring $H^*(X; \Z)$ (see, e.g., \cite{Macdonald:notes,Thomas.Yong:comin,Chaput.Perrin,Searles.Yong}).

Let $G$ be a connected complex reductive Lie group with Borel subgroup $B$ and let $B \subseteq P \subset G$ be a parabolic subgroup. Let $W$ be the Weyl group of $G$ and let $W_P \subseteq W$ be the Weyl group of the Levi subgroup of $P$. The subgroup $W_P$ is generated by the simple reflections $\{s_i\}$ indexed by the nodes corresponding to $P$ in the Dynkin diagram for $G$. The \newword{Coxeter length} of $w \in W$ is the length $\ell(w)$ of a shortest expression for $w$ as a product of simple reflections $w = s_{i_1}s_{i_2} \cdots s_{i_{\ell(w)}}$.  Every coset in $W/W_P$ has a canonical representative given by the unique element with shortest Coxeter length. We write $W^P$ for the set of minimum length representatives of the cosets $W/W_P$. For further expositions of these ideas, see, e.g., \cite{Humphreys:Lie,Humphreys}.

The \newword{generalized flag variety} $G/P$ has a Schubert cell decomposition with cells indexed by the elements of $W^P$. Here, the cells are the $B$-orbits induced from the canonical action of $B$ on $G$. We choose the convention such that the cell $e_w$ has (real) dimension $2 \ell(w)$. The generalized flag variety has a finite number of cells, all even dimensional. Moreover, there is a unique $0$-cell labeled by the %unique element $w_0^P \in W^P$ of maximal Coxeter length. 
identity element. 
Hence, the hypotheses of \autoref{thm:cohomology-of-James-reduced-product-in-general} hold and we obtain an explicit description of the cohomology ring $H^*(J(G/P); \Z)$ in terms of $H^*(G/P; \Z)$. In the cases where we have explicit combinatorial models for $H^*(G/P; \Z)$, we can leverage this understanding to obtain combinatorial models of $H^*(J(G/P); \Z)$.

\begin{example}\label{ex:flags}
	Let $G = \mathrm{GL}_n(\C)$ and take $B=P$ to be the Borel subgroup of upper triangular matrices. Then $\flags_n \coloneqq  G/P$ is the \newword{complete flag variety}, parametrizing complete flags
	\[
	V_1 \subset V_2 \subset \dots \subset V_n = \mathbb{C}^n
	\]
	of nested vector subspaces of $\C^n$. In this case, $W = W^P = S_n$ is the symmetric group on $n$ letters, so we obtain $n!$ cells in the Schubert cell decomposition of $\flags_n$, labeled by permutations. Borel \cite{Borel} gave a presentation of $H^*(\flags_n; \Z)$ as the quotient 
	\[
	H^*(\flags_n; \Z) \cong \Z[x_1, \dots, x_n] / I,
	\]
	where $I$ is the ideal generated by the $n$ \newword{elementary symmetric polynomials}
	\[
	M_{(1)}, M_{(1,1)}, \dots, M_{(1, \dots, 1)}.
	\]
	The cohomology classes $x_w$ dual to the Schubert cells $e_w$ form a $\Z$-basis of $H^*(\flags_n; \Z)$, so we identify them with elements of $\Z[x_1, \dots, x_n] / I$. Indeed, a remarkable choice of coset representatives for these classes was produced by A.~Lascoux and M.-P.~Sch\"{u}tzenberger \cite{Lascoux.Schutzenberger:Schubert} (see also \cite{Macdonald:notes}), and explained geometrically in \cite{Feher.Rimanyi, Knutson.Miller}. These representatives are the \emph{Schubert polynomials} $\mathfrak{S}_w$, indexed by the permutations $w \in S_n$. Various explicit combinatorial formulas for Schubert polynomials are known, e.g., \cite{Billey.Jockusch.Stanley,Bergeron.Billey,Weigandt.Yong}.
	
	By \autoref{thm:cohomology-of-James-reduced-product-in-general}, the cohomology $H^*(J(\flags_n); \Z)$ of the James reduced product has a cellular $\Z$-basis indexed by tuples $\Theta = (w_1, \dots, w_k)$ of nonidentity permutations in $S_n$. Considering an infinite array of variables $x_i^{(j)}$ for $1 \leq i \leq n$ and $j \in \Z_{>0}$, we obtain explicit power series representatives for these classes as follows. For $j \in \Z_{>0}$, let $\mathfrak{S}_w^{(j)}$ denote the Schubert polynomial $\mathfrak{S}_w$ in the variables $x_1^{(j)}, \dots, x_n^{(j)}$. Then, \autoref{thm:cohomology-of-James-reduced-product-in-general} gives that the cellular cohomology class $x_\Theta$ is represented in 
	\[
	H^*(J(\flags_n); \Z) \cong \sQSym(H^*(\flags_n; \Z))
	\]
	by the power series 
	\[
	M_\Theta  = \sum_{1\leq j_1 < j_2 < \dots < j_k} \mathfrak{S}_{w_1}^{(j_1)} \mathfrak{S}_{w_2}^{(j_2)} \cdots \mathfrak{S}_{w_k}^{(j_k)}.
	\]
	With this description of $M_\Theta$ in hand, it is then straightforward to compute any of the structure coefficients of $H^*(J(\flags_n))$ with respect to the cellular basis.
	\end{example}

\subsection{James reduced products of the classifying spaces $BU(k)$}\label{sec:J_classifying}
Slightly extending the above discussion, we will apply \autoref{thm:cohomology-of-James-reduced-product-in-general} to the classifying spaces $BU(k)$.
We first obtain the cell structure of $BU(k)$ from finite Grassmannians, then represent the corresponding cohomology classes by Schur functions, and finally combine the Littlewood--Richardson and Hazewinkel rules to compute the cellular structure coefficients of $H^*(J(BU(k)))$.

Let $G_n = \mathrm{GL}_n(\C)$ and let $P_n \subset G_n$ denote the maximal parabolic subgroup of block upper triangular matrices with block sizes $k$ and $n-k$. Each $G_n/P_n$ is the \newword{Grassmannian} $\mathrm{Gr}_k(\C^n)$, parametrizing (complex) $k$-dimensional linear subspaces of $\C^n$. There are natural inclusions $G_i/P_i \hookrightarrow G_{i+1}/P_{i+1}$ and $BU(k)$, the classifying space for rank $k$ complex vector bundles, is obtained as the colimit of this system.
	
	Note that these inclusion maps respect the CW structures, so we obtain a Schubert cell decomposition of $BU(k)$ induced by those of the various $\mathrm{Gr}_k(\C^n)$. An integer $1 \leq k < n$ is a \newword{descent} of the permutation $w \in S_n$ if $w(k) > w(k+1)$. A permutation is \newword{$k$-Grassmannian} if it is the identity permutation or it has $k$ as its only descent. Writing $W_n = S_n$ for the Weyl group of $G_n$, one computes that $W_n^{P_n}$ is the set of $k$-Grassmannian permutations in $S_n$. Hence, these permutations naturally index the cells of $\mathrm{Gr}_k(\C^n)$.
	
	It is more traditional, however, to index the cells of $\mathrm{Gr}_k(\C^n)$ by \newword{partitions}, i.e., nonincreasing sequences of nonnegative integers. Here, the translation from a $k$-Grassmannian permutation $w$ to a partition $\lambda$ is given by 
	\[
	w \longleftrightarrow (w(k)-k, \dots, w(2)-2,w(1)-1).
	\]
	For example, the $5$-Grassmannian permutation $w = 124693578 \in S_9$ corresponds to the partition $(4,2,1,0,0)$. We write $w_\lambda$ for the permutation associated to the partition $\lambda$. The cell $e_\lambda = e_{w_\lambda}$ has dimension $2|\lambda|$, where $|\lambda|$ denotes the sum of the elements of $\lambda$. In this notation, the cells of $\mathrm{Gr}_k(\C^n)$ are indexed by the set of all $\binom{n}{k}$ partitions of length $k$ with all elements at most $n-k$. Taking the colimit, the cells of $BU(k)$ are indexed by all partitions of length $k$ (with no upper bound on the elements). Note that there is a unique $0$-cell (indexed by the zero partition), all cells are even dimensional, and there are only finitely many cells in any fixed dimension. Hence, \autoref{thm:cohomology-of-James-reduced-product-in-general} applies to explicate the topology of $J(BU(k))$.
	
Given such a partition $\lambda$ of length $k$, the \emph{Schur polynomial} $s_\lambda \in \Z[x_1, \dots, x_k]$ is a classical symmetric polynomial arising historically from the representation theory of $S_{|\lambda|}$ and $\mathrm{GL}_k(\C)$. In our context, the important feature of Schur polynomials is that they represent the cohomology classes $x_\lambda \in H^*(BU(k))$. That is, we have
\[
x_\lambda \cdot x_\mu = \sum_{\nu} c_{\lambda,\mu}^\nu x_\nu \quad \Longleftrightarrow \quad s_\lambda \cdot s_\mu = \sum_{\nu} c_{\lambda,\mu}^\nu s_\nu;
\] here, the structure coefficients $c_{\lambda,\mu}^\nu$ are \newword{Littlewood--Richardson coefficients}, computed combinatorially by the various \emph{Littlewood--Richardson rules}, e.g., \cite{Littlewood.Richardson, Knutson.Tao.Woodward}, and both sums are over partitions of length $k$. 

To give a positive combinatorial rule for the structure coefficients of $H^*(J(BU(k)))$, we must first recall one such Littlewood--Richardson rule for multiplying Schur polynomials, as well as Hazewinkel's \cite{Hazewinkel} positive combinatorial multiplication rule for monomial quasisymmetric functions $M_\alpha$.

Given a partition $\lambda = (\lambda_1, \dots, \lambda_k)$ of length $k$, its \newword{Young diagram} is formed by placing $\lambda_i$ left-justified boxes in row $i$. For example, the Young diagram of the partition $(4,2,1)$ is $\ydiagram{4,2,1}$. We say $\lambda \subseteq \nu$ if the corresponding containment holds between their Young diagrams. In such a case, $\nu / \lambda$ denotes the set-theoretic difference $\nu \smallsetminus \lambda$ of their Young diagrams. A \newword{semistandard tableau} of shape $\nu / \lambda$ is an assignment of positive integers to the boxes of $\nu / \lambda$ such that the labels are nondecreasing left-to-right along rows and are strictly increasing top-to-bottom down columns. The \newword{content} of a semistandard tableau $T$ is the integer vector $(c_1(T), \dots)$, where $c_i(T)$ counts the number of instances of $i$ as a box label in $T$. The \newword{reading word} $w(T)$ of a semistandard tableau $T$ is the string obtained by reading the labels of $T$ by rows right-to-left and then top-to-bottom.
A semistandard tableau $T$ is \newword{ballot} if, for all $i\geq 1$, every initial segment of $w(T)$ contains at least as many instances of $i$ as of $i+1$. The following is one classical version of a Littlewood--Richardson rule.

\begin{proposition}\label{prop:LR}
	If $\lambda, \mu, \nu$ are partitions of length $k$, then the Littlewood--Richardson coefficient $c_{\lambda,\mu}^\nu$ equals the number of ballot semistandard tableaux of shape $\nu / \lambda$ that have content $\mu$. \qed
\end{proposition}

\begin{example}
	Let $\lambda = \mu = (2,1,0)$ and let $\nu = (3,2,1)$. Then by \autoref{prop:LR}, the Littlewood--Richardson coefficient $c_{\lambda,\mu}^\nu = 2$, as witnessed by the two ballot semistandard tableaux
	\[
	\ytableaushort{\blank \blank 1, \blank 1, 2} \quad \text{and} \quad \ytableaushort{\blank \blank 1, \blank 2, 1}.
	\]
	Note that the third filling 
	\[
	\ytableaushort{\blank \blank 2, \blank 1, 1}
	\] of this shape with two $1$s and one $2$ fails the ballotness condition, so does not contribute to $c_{\lambda,\mu}^\nu = 2$.
\end{example}

Finally, we can give a positive combinatorial rule for the structure coefficients of $H^*(J(BU(k)))$.
We write $s_\lambda^{(j)}$ for the Schur polynomial $s_\lambda$ in the variables $x_1^{(j)}, \dots, x_k^{(j)}$. Then \autoref{thm:cohomology-of-James-reduced-product-in-general} says that $J(BU(k))$ has cells indexed by tuples $\Lambda = (\lambda_1, \dots, \lambda_m)$ of nonzero partitions; moreover, the corresponding cellular cohomology class $x_\Lambda$ is represented by the power series
\[
M_\Lambda = \sum_{1\leq j_1 < j_2 < \dots < j_m} s_{\lambda_1}^{(j_1)} s_{\lambda_2}^{(j_2)} \cdots s_{\lambda_m}^{(j_m)}.
\]
Combining the Littlewood--Richardson rule \autoref{prop:LR} with the Hazewinkel rule \autoref{prop:M-LR}, we can compute the structure coefficients $c_{\Lambda,M}^N$ of $H^*(J(BU(k)))$ in positive combinatorial fashion.

\begin{theorem}\label{thm:BUk}
	Let $\Lambda = (\lambda_1, \dots, \lambda_\ell)$, $M = (\mu_1, \dots, \mu_m)$, and $N = (\nu_1, \dots, \nu_n)$ be tuples of nonempty partitions, each partition of length $k$. Then the structure coefficient $c_{\Lambda,M}^N$ of $H^*(J(BU(k)))$ equals the sum over order-preserving injections $\iota \colon [\ell] \to [n]$ and $\jmath \colon [m] \to [n]$ of the number of $n$-tuples $(T_1, \dots, T_n)$ of ballot semistandard tableaux, where $T_i$ has shape $\nu_i / \lambda_{\iota^{-1}(i)}$ and content $\mu_{\jmath^{-1}(i)}$. (Here, we take the convention that $\lambda_{\iota^{-1}(i)}$ is the empty partition if $i \notin \im \iota$, and similarly for $\mu_{\jmath^{-1}(i)}$.)
    That is,
    \[
c_{\Lambda,M}^N = \sum_{\substack{\iota \colon [\ell] \to [n] \\ \iota \text{ injective} \\ \iota \text{ order-preserving}}} \; \sum_{\substack{\jmath \colon [m] \to [n] \\ \jmath \text{ injective} \\ \jmath \text{ order-preserving}}} \# \left\{ (T_1, \dots, T_n) : \text{\parbox{0.38\linewidth}{%
$\forall i$ $T_i$ is ballot and semistandard\\
with shape $\nu_i / \lambda_{\iota^{-1}(i)}$ and content $\mu_{\jmath^{-1}(i)}$
}} \right\}.
    \]
\end{theorem}
\begin{proof}
	This is straightforward by combining \autoref{prop:LR} and \autoref{prop:M-LR}.
\end{proof}

\begin{example}
Set $k=3$, let $\Lambda = ((1,0,0),(2,1,0))$, and let $M = ((2,1,0))$. Then the coefficient $c_{\Lambda,M}^N$ of the class $x_N \in H^*(J(BU(3)))$ for $N = ((2,1,0),(1,0,0),(2,1,0))$ in the product $x_\Lambda \cdot x_M$ is $1$, witnessed by the injections $\iota : [2] \to [3]$ and $\jmath : [1] \to [3]$ given by $\iota(1) = 2$, $\iota(2) = 3$, and $\jmath(1) = 1$, together with the $3$-tuple of tableaux
\[
\left( 
\ytableaushort{1 1, 2}, \varnothing, \varnothing
\right),
\]
where $\varnothing$ denotes the empty tableau.

Similarly, for $N' = ((1,0,0), (2,1,0), (2,1,0))$, the coefficient $c_{\Lambda,M}^{N'}$ is easily computed to be $2$, witnessed by the  $3$-tuples of tableaux
\[
\left( 
\varnothing,
\ytableaushort{1 1, 2}, \varnothing
\right)
\quad \text{and} \quad
\left( 
\varnothing, \varnothing,
\ytableaushort{1 1, 2}
\right).
\]

A more interesting calculation is the coefficient $c_{\Lambda,\Lambda}^{N''}$ for $N'' = ((1,0,0),(1,0,0),(3,2,1))$. Here, we have 
\[
c_{\Lambda,\Lambda}^{N''} = 4,
\]
witnessed by the four $3$-tuples
\[
\left( 
\varnothing,
\ytableaushort{1}, \ytableaushort{\blank \blank 1, \blank 1, 2}
\right),
\left( 
\varnothing,
\ytableaushort{1}, \ytableaushort{\blank \blank 1, \blank 2, 1}
\right),
\left( 
\ytableaushort{1}, \varnothing, \ytableaushort{\blank \blank 1, \blank 1, 2}
\right),
\quad \text{and} \quad
\left( 
\ytableaushort{1}, \varnothing, \ytableaushort{\blank \blank 1, \blank 2, 1}
\right).
\]
\end{example}

In the same fashion, an analogous characterization of the structure coefficients of $H^*(J(X))$ can be given in any case with a correspondingly explicit combinatorial description of $H^*(X)$.

\section*{Acknowledgments}
It is a pleasure to thank Dave Anderson, Dori Bejleri, Dan Berwick-Evans, Ajneet Dhillon, Dan Edidin, Matthias Franz, Doug Park, Jenna Rajchgot, Frank Sottile, Andrew Staal, Arnav Tripathy, Jeremy Usatine, Ben Webster, and Yehao Zhou for helpful email exchanges. %OP 
We also thank Nantel Bergeron, Anders Buch, Lucas Gagnon, Allen Knutson, Thomas Lam, Andrei Okounkov, David Speyer, and Vasu Tewari for helpful conversations and encouragement. We thank the five anonymous referees for their many helpful suggestions, which significantly improved the exposition of this article.

\bibliographystyle{amsalpha}
\bibliography{quasi-symm}

\end{document}